\documentclass{amsart}

\usepackage{amsmath,
            amssymb,
            bibentry,
            enumitem,
            graphicx,
            mathrsfs,
            tabu,
            tikz,
            xypic}
\usepackage[comma,numbers,sort,square]{natbib}
\usepackage[colorlinks=true,
            linktocpage=true,
            linkcolor=magenta,
            citecolor=magenta,
            urlcolor=magenta]{hyperref}
\PassOptionsToPackage{hyphens}{url}
\usetikzlibrary{arrows,shapes,snakes}

\definecolor{grey}{rgb}{0.95,0.95,0.95}
\definecolor{green}{rgb}{0.2,0.6,0.4}

\newtheorem{theorem}{Theorem}[section]
\newtheorem{proposition}[theorem]{Proposition}
\newtheorem{lemma}[theorem]{Lemma}
\newtheorem{corollary}[theorem]{Corollary}
\theoremstyle{definition}
\newtheorem{definition}[theorem]{Definition}
\newtheorem{statement}[theorem]{Statement}
\newtheorem{question}[theorem]{Question}

\newtheoremstyle{principle}{}{}{\itshape}{}{\bfseries}{.}{.5em}{\thmnote{#3}#1}
\theoremstyle{principle}
\newtheorem*{principle}{}

\newcommand{\Psf}{\mathsf{P}}

\newcommand{\Nb}{\mathbb{N}}

\newcommand{\Qb}{\mathbb{Q}}

\newcommand{\Ccal}{\mathcal{C}}

\newcommand{\Fcal}{\mathcal{F}}

\newcommand{\Ncal}{\mathcal{N}}

\newcommand{\Qcal}{\mathcal{Q}}
\newcommand{\Rcal}{\mathcal{R}}

\newcommand{\Tcal}{\mathcal{T}}

\newcommand{\uh}{{\upharpoonright}}

\renewcommand{\setminus}{\smallsetminus}

\newcommand{\cond}[1]{\left\{\begin{array}{ll} #1 \end{array}\right.}

\newcommand{\s}[1]{\ensuremath{\sf{#1}}}

\DeclareMathOperator{\rca}{\s{RCA}_0}
\DeclareMathOperator{\stto}{\mathsf{STT}}

\DeclareMathOperator{\er}{\mathsf{ER}}
\DeclareMathOperator{\serp}{\mathsf{SER}^2_2}
\DeclareMathOperator{\erp}{\mathsf{ER}^2_2}
\DeclareMathOperator{\erps}{\mathsf{a-ER}^1_2}

\DeclareMathOperator{\wkl}{\s{WKL}_0}

\DeclareMathOperator{\rt}{\s{RT}}

\DeclareMathOperator{\tto}{\s{TT}}

\DeclareMathOperator{\coh}{\s{COH}}

\newcommand{\inter}{\mbox{int}}

\newcommand{\Leaf}{\operatorname{lvs}}

\usetikzlibrary{shapes,arrows}
\usetikzlibrary{decorations.markings}
\definecolor{lightblue}{HTML}{e6e6e6}
\definecolor{lightred}{HTML}{eca6a6}
\definecolor{lightgreen}{RGB}{164,244,140}

\newtheoremstyle{custom}
  {10pt}
  {10pt}
  {\normalfont}
  {}
  {\bfseries}
  {}
  { }
  {}

\theoremstyle{custom}

\usepackage{xcolor}	
\usepackage{soul}

\DeclareMathOperator{\res}{\upharpoonright}
\newcommand{\ran}{\operatorname{ran}}

\newcommand{\seq}[1]{\langle #1 \rangle}

\renewcommand{\iff}{\leftrightarrow}

\newcommand{\RCA}{\mathsf{RCA}}
\newcommand{\ACA}{\mathsf{ACA}}
\newcommand{\WKL}{\mathsf{WKL}}
\newcommand{\ATR}{\mathsf{ATR}}
\newcommand{\CA}{\textrm{-}\mathsf{CA}}

\newcommand{\RT}{\mathsf{RT}}
\newcommand{\COH}{\mathsf{COH}}
\newcommand{\CAC}{\mathsf{CAC}}
\newcommand{\ADS}{\mathsf{ADS}}
\newcommand{\SRT}{\mathsf{SRT}}

\newcommand{\WWKL}{\mathsf{WWKL}}

\newcommand{\TT}{\mathsf{TT}}

\newcommand{\AMT}{\mathsf{AMT}}
\newcommand{\OPT}{\mathsf{OPT}}

\newcommand{\CTT}{\mathsf{CTT}}
\newcommand{\STT}{\mathsf{STT}}

\begin{document}

\title{Coloring trees in reverse mathematics}

\author{Damir D. Dzhafarov}
\address{Department of Mathematics\\
University of Connecticut\\
Storrs, Connecticut U.S.A.}
\email{damir@math.uconn.edu}

\author{Ludovic Patey}
\address{Department of Mathematics\\
University of California, Berkeley\\
Berkeley, California U.S.A.}
\email{ludovic.patey@computability.fr}

\thanks{Dzhafarov was supported in part by NSF Grant DMS-1400267.}

\begin{abstract}
The tree theorem for pairs ($\TT^2_2$), first introduced by Chubb, Hirst, and McNicholl, asserts that given a finite coloring of pairs of comparable nodes in the full binary tree $2^{<\omega}$, there is a set of nodes isomorphic to $2^{<\omega}$ which is homogeneous for the coloring. This is a generalization of the more familiar Ramsey's theorem for pairs ($\RT^2_2$), which has been studied extensively in computability theory and reverse mathematics. We answer a longstanding open question about the strength of $\TT^2_2$, by showing that this principle does not imply the arithmetic comprehension axiom ($\ACA_0$) over the base system, recursive comprehension axiom ($\RCA_0$), of second-order arithmetic. In addition, we give a new and self-contained proof of a recent result of Patey that $\TT^2_2$ is strictly stronger than $\RT^2_2$. Combined, these results establish $\TT^2_2$ as the first known example of a natural combinatorial principle to occupy the interval strictly between $\ACA_0$ and $\RT^2_2$. The proof of this fact uses an extension of the bushy tree forcing method, and develops new techniques for dealing with combinatorial statements formulated on trees, rather than on $\omega$.
\end{abstract}

\maketitle

\section{Introduction}\label{S:intro}

Reverse mathematics is an area of mathematical logic devoted to classifying mathematical theorems according to their logical strength.
The setting for this endeavor is \emph{second-order arithmetic} which is a formal system strong enough to encompass (countable analogues of) most results of classical mathematics. It consists of the usual Peano axioms for the natural numbers, together with the \emph{comprehension scheme}, consisting of axioms asserting that the set of all~$x \in \mathbb{N}$ satisfying a given formula~$\varphi$ exists. Fragments of this system obtained by weakening the comprehension scheme are called \emph{subsystems of second-order arithmetic}. The logical strength of a theorem is then measured according to the weakest such subsystem in which that theorem can be proved. This is a two-step process: the first consists in actually finding such a subsystem, and the second in showing that the theorem ``reverses'', i.e.,~is in fact equivalent to this subsystem, over a fixed weak base system. One way to think about such a reversal is that it precisely captures the techniques needed to prove the given theorem. By extension, two theorems that turn out to be equivalent to the same subsystem (and hence to each other) can thus be regarded as requiring the same basic techniques to prove. The observation mentioned above, that most theorems can be classified into just a few categories, refers to the fact that most theorems are either provable in the weak base system, or are equivalent over it to one of four other subsystems.

The base system here is the \emph{recursive comprehension axiom} ($\RCA_0$), which restricts the comprehension scheme to~$\Delta^0_1$-definable sets. This system corresponds roughly to constructive mathematics, sufficing to prove the existence of all the computable sets, but not of any noncomputable ones. A considerably stronger system is the \emph{arithmetical comprehension axiom} ($\ACA_0$), which adds comprehension for sets definable by arithmetical formulas, i.e.,~formulas whose quantifiers only range over variables for numbers (as opposed to variables for sets of numbers). This system is considerably stronger than~$\WKL_0$, sufficing to solve the halting problem, i.e., the problem of determining whether a given computer program halts on a given input. Three other important systems are \emph{weak K\"{o}nig's lemma} ($\WKL_0$), \emph{arithmetical transfinite recursion} ($\ATR_0$), and the \emph{$\Pi^1_1$-comprehension axiom} ($\Pi^1_1\CA_0$). In order of increasing strength, these are arranged thus:
\[
\RCA_0 < \WKL_0 < \ACA_0 < \ATR_0 < \Pi^1_1\CA_0.
\]
We refer the reader to Simpson \cite{Simpson-1977} for a complete treatise on reverse mathematics, and to Soare \cite{Soare-2016} for general background on computability theory.

A striking observation, repeatedly demonstrated in the literature, is that most theorems investigated in this framework are either provable in the base system~$\RCA_0$, or else equivalent over~$\RCA_0$ to one of the other four subsystems listed above. It is from this empirical fact that these systems derive their commonly-used moniker, ``the big five''. The initial focus of the subject was almost exclusively on a kind of zoological classification of theorems according to which of the five categories they belong to. In the interval between~$\RCA_0$ and~$\ACA_0$, the study of which has received the overwhelming majority of attention in the literature, an early exception to this classification project was Ramsey's theorem for pairs. We recall its statement.

\begin{definition}\
	Fix~$X \subseteq \omega$ and~$n,k \geq 1$.
	\begin{enumerate}
		\item~$[X]^n$ is the set of all tuples~$\seq{x_0,\ldots,x_{n-1}} \in X^n$ with~$x_0 < \cdots < x_{n-1}$.
		\item A \emph{$k$-coloring of~$[X]^n$} is a function~$f : [X]^n \to \{0,\ldots,k-1\}$.
		\item A set~$Y \subseteq X$ is \emph{homogeneous for~$f$}, or \emph{$f$-homogeneous}, if there is a color~$c < k$ such that~$f ([Y]^n) = \{c\}$.
	\end{enumerate}
\end{definition}

\noindent We identify~$\{0,\ldots,k-1\}$ with~$k$, and so write a coloring simply as~$f : [X]^n \to k$. We also write~$f(x_0,\ldots,x_{n-1})$ in place of~$f(\seq{x_0,\ldots,x_{n-1}})$ for~$\seq{x_0,\ldots,x_{n-1}} \in [X]^n$.

\begin{statement}[Ramsey's theorem for~$n$-tuples and~$k$ colors ($\RT^n_k$)]
	For every coloring~$f : [\omega]^n \to k$, there is an infinite~$f$-homogeneous set.
\end{statement}

It is easy to see that over~$\RCA_0$,~$\RT^n_2$ is equivalent to~$\RT^n_k$ for any~$k \geq 2$, so in practice, we will usually restrict~$k$ to~$2$.\footnote{The situation is quite different in the closely related investigation of Ramsey's theorem under reducibilities stronger than provability in~$\RCA_0$, such as Weihrauch reducibility and computable reducibility. This analysis has gained much prominence recently as an extension of the traditional framework of reverse mathematics. (See Dorais et al.~\cite{DDHMS-2016} for an introduction, and Brattka \cite{Brattka-bib} for an updated bibliography.) In this setting, the number of colors does indeed make a difference. In particular, Patey \cite[Theorem 3.14]{Patey-TA}] has shown that if~$k > j \geq 2$ then~$\RT^n_k$ is not computably reducible to~$\RT^n_j$.} The effective analysis of Ramsey's theorem was initiated by Jockusch \cite{Jockusch-1972} in the 1970s. Recasting some of his results in the parlance of reverse mathematics, we see that~$\RCA_0$ proves~$\RT^1_2$, that~$\ACA_0$ proves~$\RT^n_2$ for every~$n$, and that for~$n \geq 3$,~$\RT^n_k$ and~$\ACA_0$ are in fact equivalent (see \cite[Theorem III.7.6]{Simpson-2009}). The situation for~$n = 2$ is different. Hirst \cite[Corollary 6.12]{Hirst-1987} showed that~$\RT^2_2$ is not provable in~$\WKL_0$ (see also \cite[Corollary 6.12]{Hirschfeldt-2014}), while much later, answering what had by then become a major question, Seetapun showed that~$\RT^2_2$ does not imply~$\ACA_0$ (see \cite[Theorems 2.1 and 3.1]{SS-1995}). Thus,~$\RT^2_2$ does not obey the ``big five'' phenomenon.

The quest to better understand~$\RT^2_2$, and in particular, of why its strength differs from that of most other theorems, has developed into a highly active and fruitful line of research, as a result of which, a ``zoo'' of mathematical principles has emerged, with a complex system of relationships between them (see \cite{Dzhafarov-zoo}). We refer the reader to Hirschfeldt \cite[Section 6]{Hirschfeldt-2014} for a survey. Figure \ref{fig1} gives a snapshot of this zoo.
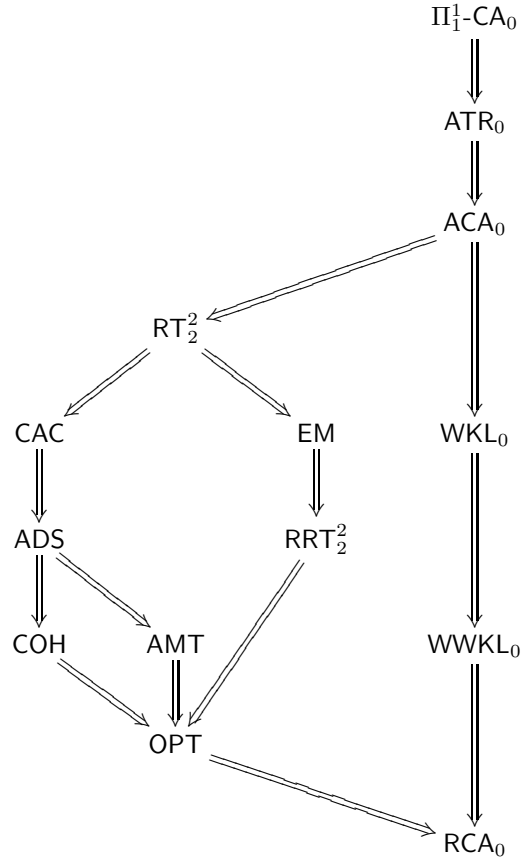
\begin{figure}[htpb]\label{fig1}
	\[
		\xymatrix @C=2pc @R=2pc{
			& & & \Pi^1_1\CA_0 \ar@2[d]\\
			& & & \ATR_0 \ar@2[d]\\
			& & & \ACA_0 \ar@2[dll] \ar@2[dd]\\
			& \RT^2_2 \ar@2[dr] \ar@2[dl]\\
			\CAC \ar@2[d] &  & \mathsf{EM} \ar@2[d] & \WKL_0 \ar@2[dd]\\
			\ADS \ar@2[dr] \ar@2[d] &  & \mathsf{RRT}^2_2 \ar@2[ddl] & \\
			\COH \ar@2[dr] & \AMT \ar@2[d] & & \WWKL_0 \ar@2[dd] \\
			& \OPT \ar@2[drr]\\
			& & & \RCA_0
		}
	\]
	\caption{Relationships among selected principles between~$\RCA_0$ and~$\ACA_0$. Arrows denote implications over~$\RCA_0$, and double arrows denote strict implications.}
\end{figure}
A conspicuous feature of this diagram is that~$\RT^2_2$ lies above every other principle, with the exception of~$\ACA_0$,~$\WKL_0$, and~$\WWKL_0$. (Whether or not~$\RT^2_2$ also implies the latter two was a longstanding problem, which was resolved only recently, by Liu \cite{Liu-2012, Liu-2015}.) While some of these principles are weaker forms of Ramsey's theorem that were introduced explicitly as a way of obtaining partial results about~$\RT^2_2$, a large number of others were studied for their own sake and with independent motivations, and come from a variety of mathematical areas (including from outside of combinatorics, such as model theory, set theory, and algorithmic randomness). 

The above suggests that~$\RT^2_2$ is a naturally strong theorem in relation to principles lying below~$\ACA_0$. Notably, there have been no examples of a natural principle\footnote{Here, we mean \emph{combinatorially} natural, and in particular, representing a single combinatorial problem. For example,~$\RT^2_2 + \WKL_0$ lies strictly below~$\ACA_0$ (see \cite[Theorem 3.1]{SS-1995}), and since~$\RT^2_2$ does not imply~$\WKL_0$, this is also strictly above~$\RT^2_2$. But~$\RT^2_2 + \WKL_0$ is a conjunction of two principles, and hence not natural in our sense.} lying strictly below~$\ACA_0$ and strictly above~$\RT^2_2$. The only candidate to be such a principle that has been previously looked at is the so-called \emph{tree theorem for pairs}, introduced by Chubb, Hirst, and McNicholl \cite{CHM-2005}, and defined below. (Recently, Patey and Frittaion \cite{FP-TA} have shown this to be closely related to a version of the Erd\H{o}s-Rado theorem. We discuss this connection further in Section \ref{S:ER} below.)

\begin{definition}[Chubb, Hirst, and McNicholl \cite{CHM-2005}]\label{defn:CHM}
	Fix~$T \subseteq 2^{<\omega}$ and~$n,k \geq 1$.
	\begin{enumerate}
		\item~$[T]^n$ is the set of all tuples~$\seq{\sigma_0,\ldots,\sigma_{n-1}} \in X^n$ with~$\sigma_0 \prec \cdots \prec \sigma_{n-1}$.
		\item A \emph{$k$-coloring of~$[T]^n$} is a function~$f : [T]^n \to k$.
		\item A set~$Y \subseteq X$ is \emph{homogeneous for~$f$}, or \emph{$f$-homogeneous}, if there is a color~$c < k$ such that~$f ([Y]^n) = \{c\}$.
		\item~$T$ is \emph{isomorphic to~$2^{<\omega}$}, written~$T \cong 2^{<\omega}$, if there is a bijection~$i : 2^{<\omega} \to T$ such that~$\sigma \preceq \tau$ if and only if~$i(\sigma) \preceq i(\tau)$ for all~$\sigma,\tau \in 2^{<\omega}$.
	\end{enumerate}
\end{definition}

\noindent As for colorings of subsets of~$\omega$, we write~$f(\sigma_0,\ldots,\sigma_{n-1})$ in place of~$f(\seq{\sigma_0,\ldots,\sigma_{n-1}})$ for~$\seq{\sigma_0,\ldots,\sigma_{n-1}} \in [T]^n$.

\begin{statement}[Tree theorem for~$n$-tuples and~$k$ colors ($\TT^n_k$)]
	For every coloring~$f : [2^{<\omega}]^n \to k$, there is an~$f$-homogeneous set~$H \subseteq 2^{<\omega}$ such that~$H \cong 2^{<\omega}$.
\end{statement}

Chubb, Hirst, and McNicholl \cite[Section 1]{CHM-2005} showed that basic results about~$\TT^n_k$ parallel the situation for Ramsey's theorem. As before, we may safely assume~$k = 2$. The base system~$\RCA_0$ suffices to prove~$\TT^1_2$, and~$\ACA_0$ suffices to prove~$\TT^n_2$ for each~$n$. It is also easy to see that~$\TT^n_2$ implies~$\RT^n_2$ over~$\RCA_0$ for all~$n$, so for~$n \geq 3$,~$\TT^n_2$ is equivalent to~$\ACA_0$, and hence to~$\RT^n_2$. By contrast, Patey \cite[Corollary 4.12]{Patey-TA3} has shown that~$\TT^2_2$ is strictly stronger than~$\RT^2_2$, leaving open the possibility that~$\TT^2_2$ might imply~$\ACA_0$. Whether this is the case has been an open question for some time, originally appearing in \cite[Question 2]{CHM-2005}, subsequently also asked about by Montalb\'{a}n \cite[Section 2.2.4]{Montalban-2011}, and variously explored by Corduan, Groszek, and Mileti \cite{CGM-2010}, Dzhafarov, Hirst, and Lakins \cite{DHL-2010}, Chong, Li, and Yang \cite{CLY-2014}, and Corduan and Groszek \cite{CG-2016}. Our main theorem in this paper is a solution to this question.

\begin{principle}[Main Theorem]
	$\RCA_0 \nvdash \TT^2_2 \to \ACA_0$.
\end{principle}

\noindent Thus, we establish~$\TT^2_2$ as the first (and only known) natural principle to lie strictly in the interval between~$\ACA_0$ and~$\RT^2_2$, and so dethrone~$\RT^2_2$ as the ``strongest of the weak principles''. Our proof of this fact develops entirely new techniques for dealing with Ramsey-like principles in the tree setting. The key difficulty here is that standard arguments for building a homogeneous set for a coloring (e.g., Mathias forcing constructions, as in the proof of Seetapun's theorem) seem, on a close inspection, very dependent on the linearity of the partial order ($\omega$) on which the coloring is defined. Thus, when dealing instead with colorings on~$2^{<\omega}$, many of the combinatorial methods from the linear setting are no longer applicable or obviously adaptable. The proof of our main theoem thus offers a way to get around these problems.

This outline of the rest of the paper is as follows. In Section \ref{S:defns}, we lay down some of the background notions and notations that we will use in the sequel. The proof of our main theorem is organized into a stable and cohesive part, in the manner first employed by Cholak, Jockusch, and Slaman \cite{CJS-2001}, though stability and cohesiveness are far less obvious notions in the tree setting than in the linear. In Section \ref{S:TT1}, we consider~$\TT^1_2$, and prove that it admits the so-called strong cone avoidance property, which then allows us to conclude that the stable tree theorem for pairs does not imply~$\ACA_0$. In Section \ref{S:TT2}, we show that the cohesive tree theorem for pairs admits cone avoidance, and hence also does not imply~$\ACA_0$. Combining these results then yields our main result. In Section \ref{S:ER}, we introduce a related theorem due to Erd\H{o}s and Rado about colorings of the rationals, and separate it from Ramsey's theorem for pairs. This separation, together with the main theorem, gives a full proof that the tree theorem for pairs lies strictly in the interval between~$\ACA_0$ and~$\RT^2_2$.

\section{Background and definitions}\label{S:defns}

%
%

Our terminology and notation in this paper is standard, e.g., as in Downey and Hirschfeldt \cite{DH-2010}. Throughout, we reserve the term \emph{tree} for downward-closed subsets of~$2^{<\omega}$, and refer to other subsets of~$2^{<\omega}$ (including those with a tree structure on them) simply as \emph{sets}.

\begin{definition}
	Let~$T \subseteq 2^{<\omega}$ be non-empty.
	\begin{enumerate}
		\item A node~$\tau \in T$ is a \emph{successor} of~$\sigma \in T$ if~$\tau \succ \sigma$ and there is no~$\xi \in T$ such that~$\sigma \prec \xi \prec \tau$.
		\item A \emph{leaf} is a node without any successor. We denote by~$\Leaf(T)$ the set of leaves of~$T$.
		\item A node~$\tau \in T$ is \emph{at level~$n$ in~$T$} if there exist precisely~$n$ proper initial segments of~$\tau \in T$.
		\item A \emph{root} of~$T$ is a node at level 0 in~$T$.
		\item The set~$T$ is \emph{at level~$n$} if every leaf is at level~$n$.
		\item We let~$T \res n$ be the set~$\{\sigma \in T : \sigma \mbox{ is at level at most } n\}$.
	\end{enumerate}
\end{definition}

\begin{definition}
	Let~$T \subseteq 2^{<\omega}$ be non-empty.
	\begin{enumerate}
		\item~$T$ is \emph{$h$-branching} for a function~$h : \omega \to \omega$ if it has a unique root and every node at level~$n$ in~$T$ which is not a leaf, has exactly~$h(n)$ immediate successors.
		\item~$T$ is \emph{2-branching} if it is~$h$-branching for the constant function~$h(n) = 2$.
		\item~$T$ is \emph{perfect} if each node has at least 2 successors.
		\item~$T$ is \emph{isomorphic to~$2^{<\omega}$}, written~$T \cong 2^{<\omega}$, if~$T$ is perfect and 2-branching.
	\end{enumerate}
\end{definition}

\noindent Note that if~$T$ is perfect then it has a~$T$-computable subset isomorphic to~$2^{<\omega}$. The definition of being isomorphic to~$2^{<\omega}$ here is different than that given in Section \ref{S:intro}, but the two are readily seen to be equivalent.

Cholak, Jockusch, and Slaman \cite[Section 7]{CJS-2001} developed a prominent framework for studying Ramsey's theorem for pairs, namely, by introducting the so-called stable Ramsey's theorem for pairs ($\SRT^2_2$) and the cohesive principle ($\COH$), into which they showed~$\RT^2_2$ can be decomposed. We recall the definitions.

\begin{definition}
	Fix~$f : [\omega]^2 \to k$ and an infinite~$X \subseteq \omega$.
	\begin{enumerate}
		\item~$f$ is \emph{stable over~$X$} if for every~$x$, there is an~$s > x$ such that~$f(x,y) = f(x,s)$ for all~$y \geq s$.
		\item~$f$ is \emph{stable} if it is stable over~$\omega$.
	\end{enumerate}
\end{definition}

\begin{statement}[Stable Ramsey's theorem for pairs and~$k$ colors ($\SRT^2_k$)]
	For every stable coloring~$f : [\omega]^2 \to k$, there is an infinite~$f$-homogeneous set. 
\end{statement}

\begin{definition}
	Let~$\vec{R} = \seq{R_0, R_1,\ldots}$ be a sequence of subsets of~$\omega$. An infinite set~$C \subseteq \omega$ is \emph{$\vec{R}$-cohesive} if for each~$i \in \omega$,~$C \subseteq^{*} R_i$ or~$C \subseteq^{*} \overline{R_i}$.
\end{definition}


\begin{statement}[Cohesive principle ($\coh$)]
	For every uniform sequence of sets~$\vec{R}$, there is a~$\vec{R}$-cohesive set.
\end{statement}

\begin{proposition}[Cholak, Jockusch, and Slaman {\cite[Lemma 7.11]{CJS-2001}}; Mileti {\cite[Claim A.1.3]{Mileti-2004}}]\label{prop:CJS_decomp}
	$\RCA_0 \vdash \SRT^2_2 + \COH$.	
\end{proposition}

\noindent The utility of this decomposition fact is that each of~$\SRT^2_2$ and~$\COH$ can be regarded as a form of~$\RT^1_2$ (the infinitary pigeonhole principle), and hence is in many ways easier to work with than~$\RT^2_2$. (See \cite[Section 6.4]{Hirschfeldt-2014} for an insightful discussion.)

For the tree theorem, notions of stability and cohesiveness were first considered by Dzhafarov, Hirst, and Lakins \cite{DHL-2010}. As it turns out, both these notions admit several reasonable adaptations from the linear to the tree setting; in \cite{DHL-2010}, the authors identified and studied five distinct such notions. For our purposes here, we will deal only with what was in \cite[Definition 3.2]{DHL-2010} called \emph{1-stability}. Since no confusion can consequently arise, we will refer to this simply as \emph{stability} below.

\begin{definition}
	Fix~$f : [2^{<\omega}]^2 \to k$ and~$T \cong 2^{<\omega}$.
	\begin{enumerate}
		\item~$f$ is \emph{stable over~$T \cong 2^{<\omega}$} if for every~$\sigma \in T$, there is a color~$c < k$ and a level~$n \in \omega$ such that~$f(\sigma, \tau) = c$ for every~$\tau \in T$ such that~$\tau \succ \sigma$ and~$|\tau| > n$.
		\item~$f$ is \emph{stable} if it is stable over~$2^{<\omega}$.
	\end{enumerate}
\end{definition}

\noindent We have the following restrictions of~$\TT^2_k$, and an associated decomposition fact.

\begin{statement}[$\STT^2_k$]
	For every stable coloring~$f : [2^{<\omega}]^2 \to k$, there is an~$f$-homogeneous set~$H \subseteq 2^{<\omega}$ such that~$H \cong 2^{<\omega}$. 
\end{statement}


\begin{statement}[$\CTT^2_k$]
	For every coloring~$f : [2^{<\omega}]^2 \to k$, there is a set~$H \subseteq 2^{<\omega}$ such that~$H \cong 2^{<\omega}$ and over which~$f$ is stable.
\end{statement}

\begin{proposition}[Dzhafarov, Hirst, and Lakins {\cite[Proposition 3.18]{DHL-2010}}]
	$\RCA_0 \vdash \TT^2_k \iff \STT^2_k + \CTT^2_k$.
\end{proposition}

\noindent The analogy between~$\SRT^2_k$ and~$\STT^2_k$ is clear. The analogy between~$\COH$ and~$\CTT^2_k$ stems from the fact, due to Hirschfeldt and Shore \cite[Proposition 4.4 and 4.8]{HS-2007}, that over~$\RCA_0 + \mathsf{B}\Sigma^0_2$ (and hence over~$\omega$-models),~$\COH$ is equivalent to the principle~$\mathsf{CRT}^2_k$, which asserts that for every coloring~$f : [\omega]^2 \to k$ there is an infinite set over which~$f$ is stable.

All the principles of the kind we are discussing here have the same syntactic form as statements in the language of second-order arithmetic, namely
\begin{equation}\label{eq:Pi12}
	(\forall X)[\phi(X) \to (\exists Y)[\theta(X,Y)]],
\end{equation}
where~$\phi$ and~$\theta$ are arithmetical formulas. The sets~$X$ satisfying~$\phi(X)$ are commonly called the \emph{instances} of the principle, and for each such~$X$, the sets~$Y$ satisfying~$\theta(X,Y)$ are called the \emph{solutions to~$X$}. (While the presentation in \eqref{eq:Pi12} is not unique, in practice there is always a fixed one we have in mind, and hence an implicitly understood class of instances and solutions for it.) For example, the instances of~$\STT^n_k$ are the stable colorings~$f : [2^{<\omega}]^2 \to k$, and the solutions to any such~$f$ are the~$f$-homogeneous sets~$H \cong 2^{<\omega}$. (See also \cite[Section 1]{DDHMS-2016} and \cite[Section 1.4]{Hirschfeldt-2014} for further examples.) The terminology is not so important for us in this paper, but it does allow us to state in a general the following properties.

\begin{definition}
	Let~$\mathsf{P}$ be a principle as in \eqref{eq:Pi12} .
	\begin{enumerate}
		\item~$\mathsf{P}$ admits \emph{cone avoidance} if for every set~$Z$, every set~$C \nleq_T Z$, and every~$Z$-computable instance~$I$ of~$\mathsf{P}$, there is a solution~$S$ to~$I$ such that~$C \nleq_T S \oplus Z$.
		\item~$\mathsf{P}$ admits \emph{strong cone avoidance} if for every set~$Z$, every set~$C \nleq_T Z$, and every instance~$I$ of~$\mathsf{P}$, there is a solution~$S$ to~$I$ such that~$C \nleq_T S \oplus Z$.
	\end{enumerate}
\end{definition}

\noindent The term ``cone avoidance'' refers to the fact that the solution~$S$ above avoids the upper cone~$\{X : X \geq_T C\}$. The adjective ``strong'' in part 2 of the definition refers to the fact that the instance~$I$ there is arbitrary, and in particular, not necessarily computable from the set~$Z$, as in part 1. The following lemma is standard.

\begin{lemma}\label{lem:coneavoidtononimplication}
	If~$\mathsf{P}$ admits cone avoidance, then~$\RCA_0 \nvdash \mathsf{P} \to \ACA_0$.
\end{lemma}

\begin{proof}
	We define a chain of sets~$Z_0 \leq_T Z_1 \leq_T \cdots$ as follows. Let~$Z_0 = \emptyset$, and suppose inductively that we have defined~$Z_n$ and that~$\emptyset' \nleq_T Z_n$. Say~$n = \seq{m,e}$, so that~$m \leq n$. If~$\Phi^{Z_m}_e$ does not define an instance of~$\mathsf{P}$, we set~$Z_{n+1} = Z_n$. Otherwise, call this instance~$I$, and regard it as a~$Z_n$-computable set. By cone avoidance for~$\mathsf{P}$, there exists a solution~$S$ to~$I$ such that~$\emptyset' \nleq_T S \oplus Z_n$. We set~$Z_{n+1} = S + Z_n$. This completes the definition of the chain. Now let~$\mathcal{M}$ be the~$\omega$-model with second order part~$\{X : (\exists n)[X \leq_T Z_n]\}$. Then by construction,~$\mathcal{M} \models \RCA_0 + \mathsf{P}$, but~$\emptyset' \notin \mathcal{M}$ and so~$\mathcal{M} \not\models \ACA_0$.
\end{proof}

Finally, we recall the definition of Mathias forcing, which is commonly employed in the construction of homogeneous sets. In the sequel, for subsets~$A$ and~$B$ of~$\omega$, we write~$A < B$ if~$A$ is finite and~$\max A < \min B$.

\begin{definition}
\
	\begin{enumerate}
		\item A \emph{Mathias condition} is a pair~$(F,X)$, where~$F$ is a finite subset of~$\omega$,~$X$ is an infinite subset of~$\omega$ called the \emph{reservoir}, and~$F < X$.
		\item A Mathias condition~$(E,Y)$ \emph{extends}~$(F,X)$, written~$(E,Y) \leq (F,X)$, if~$F \subseteq E \subseteq F \cup X$ and~$Y \subseteq X$.
	\end{enumerate}
\end{definition}

\noindent A set~$S$ is said to \emph{satisfy} a condition~$(F,X)$ if~$S$ is infinite and~$F \subseteq S \subseteq F \cup X$.

We refer the reader to Cholak, Jockusch, and Slaman \cite[Sections 4 and 5]{CJS-2001} for some prominent examples of the use of Mathias forcing in reverse mathematics, and to Cholak, Dzhafarov, Hirst, and Slaman \cite{CDHS-2014} for some general computability-theoretic facts about this forcing notion. We assume familiarity with the basics of forcing in arithmetic, as described, e.g., in Shore \cite{Shore-TA}. Below, we will work with a number of forcing notions that are defined as combinatorial elaborations of Mathias forcing, each giving rise to a forcing language and forcing relation in the standard manner (see \cite[Section 3.2]{Shore-TA}).

\section{Partitions of trees and strong cone avoidance}\label{S:TT1}

Our starting point is to prove a tree analogue of the following result about the infinitary pigeonhole principle.

\begin{theorem}[Dzhafarov and Jockusch {\cite[Lemma 3.2]{DJ-2009}}]\label{T:DzhaJo} 
	$\RT^1_2$ admits strong cone avoidance.
\end{theorem}

\noindent We prove our analogue as Theorem \ref{thm:tt1-strong-avoidance} below, from which we will also obtain cone avoidance for $\STT^2_2$.

We begin with a slightly weaker result, which we preface with a definition. Given a finite set~$F \subseteq 2^{<\omega}$, we denote by~$[F]^{\preceq}$ the set of~$\tau \in 2^{<\omega}$ extending some $\sigma \in F$. We write $[F]^{\prec}$ for~$[F]^{\preceq} \setminus F$.

\begin{definition}
	Let $T \subseteq 2^{<\omega}$ be non-empty. We denote by~$(\mathbb{P}_T, \preceq_T)$ the partial order whose elements are ordered $n$-tuples~$\seq{\sigma_0, \dots, \sigma_{n-1}}$ of pairwise incomparable nodes of~$T$, and such that~$\seq{\tau_0, \dots, \tau_{n-1}} \succeq_T \seq{\sigma_0, \dots, \sigma_{n-1}}$ if~$\tau_i \succeq \sigma_i$ for each~$i < n$. 
\end{definition}

\noindent Thus, for every finite set~$F \subseteq T$, we have that $\Leaf(F) \in \mathbb{P}_T$. Going forward, we notate elements of $\mathbb{P}_T$ as tuples $\vec{\sigma}$.

\begin{definition}
Fix a perfect set~$T \subseteq 2^{<\omega}$.
\begin{enumerate}
	\item A formula~$\varphi(U)$ (where $U$ is a finite coded set) is \emph{essential below} $\vec{\sigma} \in \mathbb{P}_T$ if for every~$\vec{\tau} \succeq_T \vec{\sigma}$, there is a finite set~$R \subseteq T \cap [\vec{\tau}]^{\preceq}$ such that~$\varphi(R)$ holds.
	\item A set~$A$ is \emph{$T$-densely $Z$-hyperimmune} if for every~$\vec{\sigma} \in \mathbb{P}_T$ and every $\Sigma^{0,Z}_1$ formula~$\varphi(U)$ essential below~$\vec{\sigma}$, there is a finite set~$R \subseteq T \cap [\vec{\sigma}]^{\preceq} \cap \overline{A}$ such that~$\varphi(R)$ holds.
\end{enumerate}
\end{definition}

\begin{theorem}\label{thm:tt1-strong-avoidance-comp}
Fix a set~$Z$, a set~$C \not \leq_T Z$ and a $Z$-computable perfect set~$T \subseteq 2^{<\omega}$.
For every set~$A \subseteq T$ which is $T$-densely $Z$-hyperimmune,
there is a set~$G \subseteq T \cap \overline{A}$ such that~$G \cong 2^{<\omega}$
and $C \not \leq_T G \oplus Z$.
\end{theorem}
\begin{proof}
Fix~$Z$, $C$, $T$ and~$A \subseteq T$.
We will build the set~$G$ by forcing. Our forcing
conditions are pairs $(F, \vec{\sigma})$, where~$F \subseteq T \cap \overline{A}$
is a finite 2-branching set and~$\vec{\sigma} \succeq_T \Leaf(F)$.
One can see the condition~$c = (F, \vec{\sigma})$ as 
the Mathias condition~$\tilde{c} = (F, T \cap [\vec{\sigma}]^{\preceq})$.
A condition~$d = (E, \vec{\tau})$ \emph{extends} a condition~$c = (F, \vec{\sigma})$
(written $d \leq c$) if the Mathias condition~$\tilde{d}$ extends the Mathias condition~$\tilde{c}$,
that is, $F \subseteq E \subseteq F \cup (T \cap [\vec{\sigma}]^{\preceq})$
and $T \cap [\vec{\tau}]^{\preceq} \subseteq T \cap [\vec{\sigma}]^{\preceq}$.
The following lemma shows that every sufficiently generic filter yields 
a set $G \cong 2^{<\omega}$.

\begin{lemma}\label{lem:tt1-strong-avoidance-comp-ext}
For every condition~$c = (F, \vec{\sigma})$ and for every~$\sigma \in \vec{\sigma}$,
there is some extension~$d = (E, \vec{\tau})$ of~$c$
such that~$E \cap [\sigma]^{\prec} \neq \emptyset$.
\end{lemma}
\begin{proof}
Fix~$c$ and~$\vec{\sigma}$.
Let~$\varphi(U)$ be the $\Sigma^{0,Z}_1$ formula
which holds if~$U$ contains at least 2 incomparable nodes in $[\sigma]^{\preceq}$.
The formula $\varphi(U)$ is essential below~$\sigma$,
so by $T$-dense $Z$-hyperimmunity of~$A$, 
there is some finite set~$R \subseteq T \cap [\vec{\sigma}]^{\preceq} \cap \overline{A}$ such that~$\varphi(R)$ holds.
Unfolding the definition of~$\varphi(R)$, there are two incomparable nodes~$\xi_0, \xi_1 \in T \cap [\sigma]^{\preceq} \cap \overline{A}$.
Let~$\vec{\tau} \succeq_T \Leaf(F \cup \{\xi_0, \xi_1\})$ be such that~$[\vec{\tau}]^{\preceq} \subseteq [\vec{\sigma}]^{\preceq}$.
The condition~$d = (F \cup \{\xi_0, \xi_1\}, \vec{\tau})$ is the desired extension of~$c$.
\end{proof}

A set~$G$ \emph{satisfies} $c = (F, \vec{\sigma})$ if $G \cong 2^{<\omega}$, $G \subseteq T \cap \overline{A}$,
and $G$ satisfies the Mathias condition~$\tilde{c}$.
set~$H$ satisfying~$c$.

\begin{lemma}\label{lem:tt1-strong-avoidance-comp-force}
For every condition~$c$ and every Turing functional~$\Gamma$,
there is an extension~$d$ of~$c$ forcing~$\Gamma^{G \oplus Z} \neq C$.
\end{lemma}
\begin{proof}
Fix $c = (F, \vec{\sigma})$ and~$\Gamma$.
Let~$\varphi(U)$ be the~$\Sigma^{0,Z}_1$ formula
which holds if there is some~$n \in \omega$
and two sets~$E_0, E_1 \subseteq U$ 
such that~$F \cup E_0$ and~$F \cup E_1$ are both 2-branching,
and $\Gamma^{(F \cup E_0) \oplus Z}(n) \neq \Gamma^{(F \cup E_1) \oplus Z}(n)$.
We have two cases.

\emph{Case 1}. $\varphi(U)$ is essential below $\vec{\sigma}$.
Since~$A$ is $T$-densely $Z$-hyperimmune, there is a finite set~$R \subseteq T \cap [\vec{\sigma}]^{\preceq} \cap \overline{A}$
such that~$\varphi(R)$ holds. Let~$E \subseteq R$
be such that $F \cup E$ is 2-branching and $\Gamma^{(F \cup E) \oplus Z}(n) \neq C(n)$.
Since~$E \subseteq [\vec{\sigma}]^{\preceq} \succeq_T \Leaf(F)$, we can define some~$\vec{\tau} \in \mathbb{P}_T$
such that~$\vec{\tau} \succeq_T \Leaf(F \cup E)$ 
and~$T \cap [\vec{\tau}]^{\preceq} \subseteq T \cap [\vec{\sigma}]^{\preceq}$.
Since~$E \subseteq T \cap \overline{A}$, the pair~$(F \cup E, \vec{\tau})$ is a valid condition
extending~$c$ and forcing~$\Gamma^{G \oplus Z} \neq C$.

\emph{Case 2}. $\varphi(U)$ is not essential below $\vec{\sigma}$. 
Let~$\vec{\tau} \succeq_T \vec{\sigma}$ be such that~$\varphi(R)$ does not hold
for every finite set~$R \subseteq T \cap [\vec{\tau}]^{\preceq}$. The condition~$d = (F, \vec{\tau})$
extends~$c$ and forces~$\Gamma^{G \oplus Z}$ to be either partial, or $Z$-computable.
\end{proof}

Let~$\mathcal{F} = \{c_0, c_1, \dots \}$ be a sufficiently generic filter,
where~$c_s = (F_s, \vec{\sigma}_s)$, and let~$G = \bigcup_s F_s$.
By the definition of a condition, $G \subseteq \overline{A}$.
By Lemma~\ref{lem:tt1-strong-avoidance-comp-ext}, $G \cong 2^{<\omega}$,
and by Lemma~\ref{lem:tt1-strong-avoidance-comp-force}, $C \not \leq_T G \oplus Z$.
This completes the proof of Theorem~\ref{thm:tt1-strong-avoidance-comp}.
\end{proof}

\begin{theorem}\label{thm:tt1-strong-avoidance}
	$\TT^1_2$ admits strong cone avoidance.
\end{theorem}
\begin{proof}
Fix a set~$Z$, a set~$C \not \leq_T Z$, and a set $A \subseteq 2^{<\omega}$. By the cone avoidance basis theorem of Jockusch and Soare~\cite[Theorem 2.5]{JS-1972a},
there is a Turing ideal~$\mathcal{M} \models \WKL$ such that~$Z \in \mathcal{M}$ and~$C \not \in \mathcal{M}$.
Suppose there is a perfect set~$X \in \mathcal{M}$ such that
$A$ is $X$-densely $X \oplus Z$-hyperimmune. By
Theorem~\ref{thm:tt1-strong-avoidance-comp}, there is a set~$G \subseteq X \cap \overline{A}$
such that $G \cong 2^{<\omega}$ and~$C \not \leq_T G \oplus \tilde{Z}$, so \emph{a fortiori} $C \not \leq_T G \oplus Z$,
in which case we are done. Suppose now there is no such set~$X \in \mathcal{M}$.

We will build our set~$G$ by forcing. Our forcing conditions are Mathias conditions~$(F, X)$,
where~$F \subseteq A$ is a finite 2-branching set, 
$X \in \mathcal{M}$, and $F \cup X$ is a perfect set. The extension is the usual Mathias extension.

\begin{lemma}\label{lem:tt1-strong-avoidance-ext}
For every condition~$c = (F, X)$ and for every leaf~$\sigma$ of~$F$,
there is some extension~$d = (E, Y)$ of~$c$
such that~$E \cap [\sigma]^{\prec} \neq \emptyset$.
\end{lemma}
\begin{proof}
Fix~$c$ and~$\sigma$. Since $F \cup X$ is perfect, 
so is $X_\sigma = X \cap [\sigma]^{\preceq}$.
Moreover, $X_\sigma\in \mathcal{M}$, 
so by assumption, $A$ is not $X_\sigma$-densely $X_\sigma \oplus Z$-hyperimmune.
Unfolding the definition, there is a $\Sigma^{0,X_\sigma \oplus Z}_1$ formula~$\varphi(U)$ 
essential below some~$\vec{\tau} \in \mathbb{P}_{X_\sigma}$,
such that~$R \cap A \neq \emptyset$ whenever $\varphi(R)$ holds and $R \subseteq X \cap [\vec{\tau}]^{\preceq}$.
Let~$\vec{\xi}_0 \succeq_{X_\sigma} \vec{\tau}$ and~$\vec{\xi}_1 \succeq_{X_\sigma} \vec{\tau}$
be such that~$[\vec{\xi}_0]^{\preceq} \cap [\vec{\xi}_1]^{\preceq} = \emptyset$.
By essentiality of~$\varphi(U)$, there are two finite sets $R_0 \subseteq X \cap [\vec{\xi}_0]^{\preceq}$
and~$R_1 \subseteq X \cap [\vec{\xi}_1]^{\preceq}$ such that~$\varphi(R_0)$ and~$\varphi(R_1)$ hold.
In particular, we can pick some~$\rho_0 \in R_0 \cap A$ and~$\rho_1 \in R_1 \cap A$.
Note that $\rho_0, \rho_1 \in X \cap [\vec{\tau}]^{\preceq} \subseteq X \cap [\sigma]^{\preceq}$,
and therefore $F \cup \{\rho_0, \rho_1\}$ is 2-branching.
Let~$Y \subseteq X$ be obtained by removing finitely many elements, so that $(F \cup \{\rho_1, \rho_1\}, Y)$ is a Mathias condition.
Since $\rho_0, \rho_1 \in X$ and $F \cup X$ is perfect, so is $F \cup \{\rho_0, \rho_1\} \cup Y$.
Therefore, the condition~$d = (F \cup \{\rho_0, \rho_1\}, Y)$ is the desired extension of~$c$.
\end{proof}

We say a set~$G$ \emph{satisfies} $c = (F, X)$ if $G \cong 2^{<\omega}$, $G \subseteq A$,
and $G$ satisfies $c$ as a Mathias condition.

\begin{lemma}\label{lem:tt1-strong-avoidance-force}
For every condition~$c$ and every Turing functional~$\Gamma$,
there is an extension~$d$ of~$c$ forcing~$\Gamma^{G \oplus Z} \neq C$.
\end{lemma}
\begin{proof}
Fix $c = (F, X)$ and~$\Gamma$.
For every~$\sigma \in \Leaf(F)$, the set $X_\sigma = X \cap [\sigma]^{\preceq}$ is perfect 
and belongs to $\mathcal{M}$, so by assumption, $A$ is not $X_\sigma$-densely $X_\sigma \oplus Z$-hyperimmune.
Unfolding the definition, there is a $\Sigma^{0,X_\sigma \oplus Z}_1$ formula~$\varphi_\sigma(U)$ 
essential below some~$\vec{\tau}_\sigma \in \mathbb{P}_{X_\sigma}$,
such that~$R \cap A \neq \emptyset$ whenever $\varphi_\sigma(R)$ holds and $R \subseteq X_\sigma \cap [\vec{\tau}_\sigma]^{\preceq}$.
Fix a $X \oplus Z$-computable enumeration~$\vec{\xi}_0, \vec{\xi}_1, \dots$ of all $\vec{\xi} \in \mathbb{P}_{X_\sigma}$
such that~$\vec{\xi} \succeq_{X_\sigma} \vec{\tau}_\sigma$. For each~$i \in \omega$,
let~$R_i \subseteq X \cap [\vec{\xi}_i]^{\preceq}$ be such that $\varphi_\sigma(R_i)$ holds.
Note that~$R_i \cap A \neq \emptyset$ by assumption.
Define the $\Pi^{0,X \oplus Z}_1$ class~$\mathcal{C}_\sigma$ of all $P \in X^{\omega}$
such that $(\forall i \in \omega)P(i) \in R_i$. In particular,
by choice of the~$R$'s, there is some~$P \in \mathcal{C}_\sigma$
such that~$C(P) \subseteq A$. Moreover, by the usual pairing argument,
for every $P \in \mathcal{C}_\sigma$, there is some~$\tau \in \vec{\tau}_\sigma$ such that~$\ran(P)$ is dense in~$X \cap [\tau]^{\preceq}$.

Let~$\mathcal{C}$ be the $\Pi^{0,X \oplus Z}_1$ class of all $\seq{P_\sigma : \sigma \in \Leaf(F)}$,
where~$P_I \in \mathcal{C}_I$ such that, for every pair of finite sets 
$E_0, E_1 \subseteq \bigcup_{\sigma \in \Leaf(F)} \ran(P_\sigma)$
with~$F \cup E_0$ and~$F \cup E_1$ both 2-branching, 
it is not the case that $\Gamma^{(F \cup E_0) \oplus Z} \downarrow = \Gamma^{(F \cup E_1) \oplus Z} \downarrow$.
We have two cases.

\emph{Case 1}. $\mathcal{C}$ is empty. For each~$\sigma \in \Leaf(F)$,
let~$P_\sigma \in \mathcal{C}_\sigma$ be such that~$\ran(P_\sigma) \subseteq A$.
In particular, $\seq{P_\sigma : \sigma \in \Leaf(F)} \not \in \mathcal{C}$,
so by definition of~$\mathcal{C}$, there is a finite set $E \subseteq \bigcup_{\sigma \in \Leaf(F)} \ran(P_\sigma)$
such that~$F \cup E$ is 2-branching,
and some~$n \in \omega$ such that $\Gamma^{(F \cup E) \oplus Z}(n) \neq C(n)$.
Let~$Y \subseteq X$ be obtained by removing finitely elements, so that $(F \cup E, Y)$ is a valid Mathias condition.
Since $E \subseteq X$ and $F \cup X$ is perfect, so is $F \cup E \cup Y$.
The condition~$d = (F \cup E, Y)$ is an extension of~$c$
forcing~$\Gamma^{G \oplus Z} \neq C$.

\emph{Case 2}. $\mathcal{C}$ is non-empty. Since $\mathcal{M} \models \WKL$,
there is some~$\seq{P_\sigma : \sigma \in \Leaf(F)} \in \mathcal{C} \cap \mathcal{M}$.
For each~$\sigma \in \Leaf(F)$, let~$\tau_\sigma \in \vec{\tau}_\sigma$
be such that $\ran(P_\sigma)$ is dense in $X \cap [\tau_\sigma]^{\preceq}$.
Let~$Y$ be obtained by $\vec{P} \oplus Z$-computably thinning out the set~$X$
so that $X \cap [\tau_\sigma]^{\preceq}$ is perfect for each~$\sigma \in \Leaf(F)$.
Note that~$Y \in \mathcal{M}$.
The condition~$d = (F, Y)$ is an extension of~$c$ forcing $\Gamma^{G \oplus Z}$
to be either partial, or $X \oplus Z$-computable.
\end{proof}

Let~$\mathcal{F} = \{c_0, c_1, \dots \}$ be a sufficiently generic filter,
where~$c_s = (F_s, X_s)$, and let~$G = \bigcup_s F_s$.
By the definition of a condition, $G \subseteq A$.
By Lemma~\ref{lem:tt1-strong-avoidance-ext}, $G \cong 2^{<\omega}$,
and by Lemma~\ref{lem:tt1-strong-avoidance-force}, $C \not \leq_T G \oplus Z$.
This completes the proof of Theorem~\ref{thm:tt1-strong-avoidance}.
\end{proof}

For general interest, we note the following immediate consequence of Theorem \ref{thm:tt1-strong-avoidance}, which may be considered a first step in the direction of proving our main theorem. However, in the proof of the main theorem, we will actually need the full version of Theorem \ref{thm:tt1-strong-avoidance} rather than merely this corollary.

\begin{corollary}\label{cor:stt22coneavoid}
	$\STT^2_2$ admits cone avoidance. Hence, $\RCA_0 \nvdash \STT^2_2 \to \ACA_0$.
\end{corollary}

\begin{proof}
	Consider any set $Z$, any $Z$-computable stable $f : [2^{<\omega}]^2 \to 2$, and any $C \nleq_T Z$. For each $c < 2$, let $A_c$ be the set of all $\sigma \in 2^{<\omega}$ such that $(\exists n)(\forall \tau)[(\tau \succeq \sigma \wedge |\tau| \geq n) \to f(\sigma,\tau) = c$. By stability of $f$, $A_0 = \overline{A_1}$. By Theorem \ref{thm:tt1-strong-avoidance}, there exists a $c < 2$ and a $G\subseteq A_c$ such that $G \cong 2^{<\omega}$ and $C \nleq_T G \oplus Z$. Now $G$ can be $(G \oplus Z)$-computably pruned to $H \subseteq G$ such that $H \cong 2^{<\omega}$ and $f(\sigma,\tau) = c$ for all $\seq{\sigma,\tau} \in [H]^2$. And we have $H \oplus Z \leq_T G \oplus Z$, hence $C \nleq_T H \oplus Z$. The rest of the corollary now follows by Lemma \ref{lem:coneavoidtononimplication}.
\end{proof}


\section{The tree theorem for pairs and cone avoidance}\label{S:TT2}

Our goal in this section is to prove our main theorem, which we will do in the following more specific form.

\begin{theorem}\label{thm:tt22-cone-avoidance}
$\TT^2_2$ admits cone avoidance. Hence,~$\RCA_0 \nvdash \TT^2_2 \to \ACA_0$.
\end{theorem}

\noindent In order to prove the theorem,
we need to introduce an adaptation of the bushy tree forcing framework. Bushy tree forcing was developed by Kumabe \cite{Kumabe-1996} and Kumabe and Lewis \cite{KL-2009} and has been employed to prove a number of results in algorithmic randomness and classical computability theory, particularly to do with diagonally-noncomputable functions. We refer the reader to Khan and Miller \cite{KM-TA} for a survey on some of these results, along with a primer on bushy tree forcing as it is used to prove them.

The use of this forcing for the purposes of studying combinatorial principles like Ramsey's theorem is more recent, with some early examples by Patey \cite{MR3416155, Patey-TA, Patey-TA4}. Our treatment here will be self-contained.

\begin{definition}
Given two sets~$T, S \subseteq 2^{<\omega}$,
we write~$S \lhd T$ if~$S \subseteq T$ and whenever~$\tau \in S$
and~$\sigma$ is a proper initial segment of~$\tau$ in~$T$, then~$\sigma \in S$.
We say that a set~$B \subseteq T$ is \emph{$h$-big in~$T$} for some function~$h$ if
there is an~$h$-branching set~$S \lhd T$ such that~$\Leaf(S) \subseteq B$.
\end{definition}

In particular, if~$S \lhd T$, then any node at level~$n$ in~$S$ is at level~$n$ in~$T$.
Also note that relation~$\lhd$ is transitive.
The following lemma is a standard combinatorial fact about bushy tree forcing (see \cite{KM-TA}, Lemma 2.4). As our framework is slightly different from the general one, we provide a proof for the sake of completeness.

\begin{lemma}
\label{lem:smallness-additivity}
Fix~$T \subseteq 2^{<\omega}$.
If~$B \cup C$ is~$(h+g-1)$-big in~$T$, then either~$B$ is~$h$-big in~$T$ or~$C$ is~$g$-big in~$T$.
\end{lemma}
\begin{proof}
Fix an~$(h+g-1)$-branching set~$S \lhd T$ such that~$\Leaf(S) \subseteq B \cup C$. We label each~$\sigma \in S$ by either~$B$ or~$C$, as follows. Label each~$\sigma \in \Leaf(S)$ by~$B$ if~$\sigma \in B$, and by~$C$ if~$\sigma \notin B$ (in which case, of course,~$\sigma \in C$). Now fix~$\sigma \in S \setminus \Leaf{S}$, and assume by induction that every successor of~$\sigma$ in~$S$ has already been labeled. Say~$\sigma$ is at level~$n$. If at least~$h(n)$ successors of~$\sigma$ are labeled by~$B$, then label~$\sigma$ by~$B$ as well. Otherwise, label~$\sigma$ by~$C$, and notice that as~$S$ is~$(h+g-1)$-branching, this means at least~$g(n)$ successors of~$\sigma$ in~$S$ are labeled by~$C$. This completes the labeling. We now define~$S' \lhd T$ as follows. Add the root of~$S$ to~$S'$. Having added~$\sigma$ to~$S'$, add to~$S'$ either the least~$h(n)$ successors of~$\sigma$ labeled by~$B$ or the least~$g(n)$ successors labeled by~$C$, depending as~$\sigma$ is itself labeled~$B$ or~$C$, respectively. Then~$S'$ witnesses that either~$B$ is~$h$-big in~$T$ or~$C$ is~$g$-big in~$T$, as desired.
\end{proof}

For every~$k, p$, let~$h_{k,p}$ be the function defined by induction over~$n$ by
$h_{k,p}(0) = kp-1$ and~$h_{k,p}(n+1) = h_{k,kp-1}(n)$.
The function~$h_{k,p}$ has been designed so that it satisfies the following combinatorial lemma.

\begin{lemma}\label{lem:finite-branching-combi}
Fix some~$k, p \geq 1$. Let~$T \subseteq 2^{<\omega}$ be an~$h_{k,p}$-branching set at level~$n$,
and~$g : [T]^2 \to k$ be a coloring.
There is a~$p$-branching set~$S \lhd T$ at level~$n$ such that
\[
(\forall \sigma \in S \setminus \Leaf(S))(\exists c < k)(\forall \tau \in \Leaf(S))[\sigma \prec \tau \rightarrow g(\sigma, \tau) = c].
\]
\end{lemma}
\begin{proof}
By induction over~$n$. The case~$n = 0$ is vacuously true.
Let~$T \subseteq 2^{<\omega}$ be an~$h_{k,p}$-branching set at level~$n+1$
and~$g : [T]^2 \to k$ be a coloring. For each node~$\xi$ at level 1 in~$T$,
let~$T_\xi = T \cap [\xi]^{\preceq}$ and~$g_\xi : [T_\xi]^2 \to k$
be the restriction of~$g$ over~$T_\xi$. Note that every node at level~$m$ in~$T_\xi$
is at level~$m+1$ in~$T$. In particular, if it is not a leaf in~$T_\xi$, then
it has~$h_{k,p}(m+1) = h_{k,kp-1}(m)$ immediate successors in~$T_\xi$.
Therefore,~$T_\xi$ is a~$h_{k, kp-1}$-branching set at level~$n$.
By induction hypothesis, there is a~$(kp-1)$-branching set~$S_\xi \lhd T_\xi$ at level~$n$
such that 
\[
(\forall \sigma \in S_\xi \setminus \Leaf(S_\xi))(\exists c < k)
	(\forall \rho \in \Leaf(S_\xi))[\sigma \prec \rho \rightarrow g(\sigma, \rho) = c].
\]
Note that since~$T$ is~$h_{k,p}$-branching, there are~$h_{k,p}(0) = kp-1$ nodes at level 1 in~$T$,
so~$T_1 = \{\epsilon\} \cup \bigcup_{\xi} S_\xi$ is a~$(kp-1)$-branching set at level~$n+1$,
where~$\epsilon$ is the root of~$T$. Moreover,~$T_1 \lhd T$. 
For each~$i < k$,~$B_i = \{\tau \in \Leaf(T_1) : g(\epsilon, \tau) = i\}$.
The set~$B_0 \cup \dots \cup B_{k-1}$ is~$(kp-1)$-big in~$T_1$, so by Lemma~\ref{lem:smallness-additivity}, 
there is some~$i < k$ such that~$B_i$ is~$p$-big in~$T_1$.
Then by definition of~$p$-bigness of~$B$, there is a~$p$-branching set~$S \lhd T_1 \lhd T$
such that~$\Leaf(S) \subseteq B_i$. We claim that~$S$ satisfies the desired property.
Fix some~$\sigma \in S \setminus \Leaf(S)$. We have two cases.

In the first case,~$\sigma = \epsilon$. Let~$c = i$.
Since~$\Leaf(S) \subseteq B_i$, then for each~$\tau \in \Leaf(S)$ such that~$\sigma \prec \tau$,
$f(\epsilon, \tau) = i = c$.

In the second case,~$\sigma \in T_\xi$ for some~$\xi$ at level 1 in~$T$.
Let~$c < k$ be such that~$(\forall \rho \in \Leaf(S_\xi))[\sigma \prec \rho \rightarrow g(\sigma, \rho) = c$.
This color exists by definition of~$S_\xi$. For every~$\rho \in \Leaf(S)$ such that~$\sigma \prec \rho$,
$\rho \in S_\xi$, so~$f(\sigma, \rho) = c$.
This completes the proof of Lemma~\ref{lem:finite-branching-combi}.
\end{proof}

\begin{lemma}\label{lem:infinite-branching-combi}
Fix some~$k \geq 1$, a set~$Z$, a~$Z$-computable~$h_{k,2}$-branching perfect set~$T \subseteq 2^{<\omega}$
and a~$Z$-computable coloring~$f : [2^{<\omega}]^2 \to k$.
For every~$n$, there is a~$2$-branching set~$S \lhd T \res n$ at level~$n$
and a~$Z$-computable set~$X \subseteq T$ such that~$S \cup X$ is perfect and
\[
(\forall \sigma \in S)(\exists c < k)(\forall \rho \in X)[\sigma \prec \rho \rightarrow f(\sigma, \rho) = c].
\]
\end{lemma}
\begin{proof} 
Fix~$n$. For each~$\tau \in \Leaf(T \res n)$, let~$\sigma_0, \sigma_1, \dots, \sigma_n$ be the initial segments
of~$\tau$ in~$T$, with~$\sigma_n = \tau$.
Apply~$\TT^1_{k^{n+1}}$ over~$T \cap [\tau]^{\prec}$ 
with the coloring~$\rho \mapsto \langle f(\sigma_0,\rho), \dots, f(\sigma_n, \rho) \rangle$
to obtain a~$Z$-computable set~$X_\tau \subseteq T \cap [\tau]^{\prec}$ such that~$\{\tau\} \cup X_\tau$ is perfect
and~$X_\tau$ is homogeneous for some color~$\langle c^\tau_{\sigma_0}, \dots, c^\tau_{\sigma_n} \rangle$.
Let~$g : [T \res n]^2 \to k$ be the coloring defined by
\[
g(\sigma, \tau) = \cond{ c^\tau_\sigma & \mbox{ if } \tau \in \Leaf(T \res n)\\
	0 & \mbox{ otherwise }
}
\]
By Lemma~\ref{lem:finite-branching-combi} applied to~$T \res n$ and~$g$,
there is a~$2$-branching set~$S \lhd T \res n$ at level~$n$ such that 
\[
(\forall \sigma \in S \setminus \Leaf(S))(\exists c < k)(\forall \tau \in \Leaf(S))[\sigma \prec \tau \rightarrow g(\sigma, \tau) = c]
\]
Let~$X = \bigcup_{\tau \in \Leaf(S)} X_\tau$.
Since~$X_\tau \subseteq [\tau]^\prec$ and~$\{\tau\} \cup X_\tau$ is perfect for every~$\tau \in \Leaf(S)$,
then~$S \cup X$ is perfect. We claim that~$S$ and~$X$ satisfy the desired property. 
Fix some~$\sigma \in S$. We have two cases. 

In the first case,~$\sigma \in \Leaf(S)$. Set~$c = c^\sigma_\sigma$.
For every~$\rho \in X$ such that~$\sigma \prec \rho$,~$\rho \in X_\sigma$.
By definition of~$X_\sigma$,~$f(\sigma, \rho) = c^\sigma_\sigma = c$.

In the second case,~$\sigma$ is a proper initial segment of some~$\tau \in \Leaf(S)$.
Set~$c = c^\tau_\sigma$. For every~$\rho \in X$ such that~$\sigma \prec \rho$, 
there is some~$\tau_1 \in \Leaf(S)$ such that~$\rho \in X_{\tau_1}$.
By definition of~$S$, there is some color~$c_1$ such that
$(\forall \tau_2 \in \Leaf(S)][\sigma \prec \tau_2 \rightarrow g(\sigma, \tau_2)] = c_1$.
By letting~$\tau_2 = \tau$, we obtain~$c_1 = c^\tau_\sigma$, and by letting~$\tau_2 = \tau_1$,
we obtain~$c_1 = c^{\tau_1}_\sigma$. It follows that~$c = c^{\tau_1}_\sigma$,
and by definition of~$X_{\tau_1}$,~$f(\sigma, \rho) = c^{\tau_1}_\sigma = c$.
This completes the proof of Lemma~\ref{lem:infinite-branching-combi}.
\end{proof}

Our next result is the weaker form of Theorem \ref{thm:tt22-cone-avoidance} for~$\CTT^2_2$. Notice that together with Corollary \ref{cor:stt22coneavoid}, this gives us cone avoidance for both~$\STT^2_2$ and~$\CTT^2_2$. However, this is by itself not enough to yield Theorem \ref{thm:tt22-cone-avoidance}, as it does not imply cone avoidance for the join,~$\STT^2_2 + \COH$. We will prove that this is the case at the end of the section.

\begin{theorem}\label{thm:ctt22-cone-avoidance}
$\CTT^2_2$ admits cone avoidance.
\end{theorem}
\begin{proof}
Fix a set~$Z$, a set~$C \not \leq_T Z$, and a~$Z$-computable
coloring~$f : [2^{<\omega}]^2 \to 2$.
We will construct a set~$G \subseteq 2^{<\omega}$ over which~$f$ is stable,
such that~$G \cong 2^{<\omega}$ and~$C \not \leq_T G \oplus Z$.
The set~$G$ will be constructed by a forcing
whose conditions are Mathias conditions~$(F, X)$, where~$F$ is a finite 2-branching set,
$C \not \leq_T X \oplus Z$, and~$F \cup X$ is a perfect set. Moreover, we require that
\[
(\forall \sigma \in F)(\exists c < 2)(\forall \tau \in X)[\sigma \prec \tau \rightarrow f(\sigma, \tau) = c]
\]
The extension is the usual Mathias extension. The following lemmas shows that every sufficiently
generic filter for this notion of forcing yields a set~$G \cong 2^{<\omega}$.

\begin{lemma}\label{lem:ctt22-cone-avoidance-ext}
For every condition~$c = (F, X)$ and for every leaf~$\sigma$ of~$F$,
there is some extension~$d = (E, Y)$ of~$c$
such that~$E \cap [\sigma]^{\prec} \neq \emptyset$.
\end{lemma}
\begin{proof}
Fix some leaf~$\sigma \in F$. Since~$F \cup X$ is perfect,
we can pick three pairwise-incomparable nodes~$\xi_0, \xi_1, \xi_2 \in X \cap [\sigma]^\prec$.
In particular,~$X \cap [\xi_i]^\prec$ is perfect for each~$i < 3$,
so by applying~$\TT^1_2$ to~$X \cap [\xi_i]^\prec$ with~$\rho \mapsto f(\xi_i, \rho)$,
we obtain an~$X \oplus Z$-computable set~$X_i \subseteq X \cap [\xi_i]^\prec$ and a color~$c_i < 2$
such that~$\{\xi_i\} \cup X_i$ is perfect, and~$f(\xi_i, \rho) = c_i$ for each~$\rho \in X_i$.
By the pigeonhole principle, there is some~$c < 2$ and some~$i_0 < i_1 < 3$ such that~$c = c_{i_0} = c_{i_1}$.
By removing finitely many elements of~$(X \setminus [\sigma]^\prec) \cup X_{i_0} \cup X_{i_1}$,
we obtain a~$Z$-computable set~$Y$ such that
the condition~$d = (F \cup \{\xi_{i_0}, \xi_{i_1}\}, Y)$ is a valid extension of~$c$.
\end{proof}

We say a set~$S$ \emph{satisfies}~$c = (F, X)$ if~$S \cong 2^{<\omega}$
and~$S$ satisfies~$c$ as a Mathias condition.
set~$H$ satisfying~$c$.

\begin{lemma}\label{lem:ctt22-cone-avoidance-force}
For every condition~$c$ and every Turing functional~$\Gamma$,
there is an extension~$d$ of~$c$ forcing~$\Gamma^{G \oplus Z} \neq C$.
\end{lemma}
\begin{proof}
Fix~$c = (F, X)$ and~$\Gamma$.
For every~$\xi \in \Leaf(F)$, let~$X_\xi$ be an~$X \oplus Z$-computable
$h_{2,2}$-branching perfect subtree of~$X \cap [\xi]^{\prec}$.
Let~$\mathcal{C}$ be the~$\Pi^{0,X \oplus Z}_1$ class of all~$\langle S_\xi : \xi \in \Leaf(F) \rangle$
such that~$S_\xi \lhd X_\xi$ is a perfect 2-branching set for each~$\xi \in \Leaf(F)$.
Let~$\mathcal{D}$ be the~$\Pi^{0,X \oplus Z}_1$ class of all~$\langle S_\xi : \xi \in \Leaf(F) \rangle \in \mathcal{C}$
such that for every pair~$E_0, E_1 \subseteq \bigcup_\xi S_\xi$ of finite sets with 
$F \cup E_0$ and~$F \cup E_1$ both 2-branching, 
it is not the case that~$\Gamma^{(F \cup E_0) \oplus Z} \downarrow = \Gamma^{(F \cup E_1) \oplus Z} \downarrow$.
We have two cases.

\emph{Case 1}.~$\mathcal{D}$ is empty. By compactness, there is some~$n$
such that for every ~$\langle S_\xi : \xi \in \Leaf(F) \rangle$
where~$S_\xi \lhd X_\xi \res n$ is a 2-branching set at level~$n$,
there is a set~$E \subseteq \bigcup_\xi S_\xi$ and some~$m$ such that 
$F \cup E$ are is 2-branching, and~$\Gamma^{(F \cup E) \oplus Z}(m) \neq C(m)$.
By Lemma~\ref{lem:infinite-branching-combi} applied to each~$X_\xi$, 
there are some ~$\langle S_\xi : \xi \in \Leaf(F) \rangle$
where~$S_\xi \lhd X_\xi \res n$ is a 2-branching set at level~$n$,
and some~$Z$-computable sets~$\langle Y_\xi \subseteq X_\xi : \xi \in \Leaf(F) \rangle$
such that~$S_\xi \cup Y_\xi$ is perfect and
\[
(\forall \sigma \in S_\xi)(\exists c < 2)(\forall \rho \in Y_\xi)[\sigma \prec \rho \rightarrow f(\sigma, \rho) = c].
\]
Let~$Y \subseteq \bigcup_\xi Y_\xi$ be obtained by removing finitely elements, 
so that~$(F \cup E, Y)$ is a valid Mathias condition
and~$F \cup E \cup Y$ is perfect. The condition~$d = (F \cup E, Y)$ is an extension of~$c$
forcing~$\Gamma^{G \oplus Z} \neq C$.

\emph{Case 2}.~$\mathcal{D}$ is non-empty. By the cone avoidance basis theorem, there is some~$\vec{S} = \langle S_\xi : \xi \in \Leaf(F) \rangle \in \mathcal{D}$
such that~$C \not \leq_T \vec{S} \oplus Z$.
Let~$Y$ be obtained by removing finitely many elements of~$\bigcup_\xi S_\xi$
so that~$F \cup Y$ is perfect. The condition~$d = (F, Y)$ is an extension of~$c$ forcing~$\Gamma^{G \oplus Z}$
to be either partial, or~$X \oplus Z$-computable.
\end{proof}

Let~$\mathcal{F} = \{c_0, c_1, \dots \}$ be a sufficiently generic filter,
where~$c_s = (F_s, X_s)$, and let~$G = \bigcup_s F_s$.
By definition of a condition,~$f$ is stable over~$G$.
By Lemma~\ref{lem:ctt22-cone-avoidance-ext},~$G \cong 2^{<\omega}$,
and by Lemma~\ref{lem:ctt22-cone-avoidance-force},~$C \not \leq_T G \oplus Z$.
This completes the proof of Theorem~\ref{thm:ctt22-cone-avoidance}. \qedhere
\end{proof}

We are now ready to prove that~$\TT^2_2$ admits cone avoidance.

\begin{proof}[Proof of Theorem~\ref{thm:tt22-cone-avoidance}]
Fix a set~$Z$, a set~$C \not \leq_T Z$, and a~$Z$-computable
coloring~$f : [2^{<\omega}]^2 \to 2$.
By cone avoidance of~$\CTT^2_2$ (Theorem~\ref{thm:ctt22-cone-avoidance}), 
there is a set~$H_0 \subseteq 2^{<\omega}$ such that~$H_0 \cong 2^{<\omega}$,
$C \not \leq_T H_0 \oplus Z$, and over which~$f$ is stable.
Define the~$\Delta^{0,H_0 \oplus Z}_2$ set~$A \subseteq H_0$ by
\[
A = \{ \sigma \in H_0 : (\exists n)(\forall \tau \in H_0)[(\tau \succ \sigma \wedge |\tau| \geq n) \rightarrow f(\sigma, \tau) = 1 \}.
\]
By strong cone avoidance of~$\TT^1_2$ (Theorem~\ref{thm:tt1-strong-avoidance}),
there is a subset~$H_1$ of~$A$ or~$H_0 \setminus A$ such that~$H_1 \cong 2^{<\omega}$ and~$C \not \leq_T H_1 \oplus H_0 \oplus Z$.
Note~$H_1 \oplus H_0 \oplus Z$ computes an~$f$-homogeneous set~$H \subseteq H_1$
such that~$H \cong 2^{<\omega}$. In particular,~$C \not \leq_T H \oplus Z$.
\end{proof}

\section{An Erd\H{o}s-Rado theorem and Ramsey's theorem for pairs}\label{S:ER}

Among the candidate statements between Ramsey's theorem for pairs and~$\ACA_0$, 
a theorem from Erd\H{o}s and Rado~\cite[Theorem 4, p. 427]{Erdos1952Combinatorial} is arguably the most natural.
This theorem extends Ramsey's theorem for pairs to coloring over pairs of rationals. In what follows, we will be dealing with colorings on pairs of rational numbers. Let~$\leq_{\mathbb{Q}}$ be the standard ordering of the rationals.

\begin{definition}\
	Fix~$X \subseteq \mathbb{Q}$ and~$n,k \geq 1$.
	\begin{enumerate}
		\item~$[X]^2$ is the set of all pairs~$\seq{x,y} \in X^2$ with~$x < y$.
		\item A \emph{$k$-coloring of~$[X]^2$} is a function~$f : [X]^n \to k$, and we write~$f(x,y)$ instead of~$f(\seq{x,y})$ for~$\seq{x,y} \in [X]^2$.
		\item A set~$Y \subseteq X$ is \emph{homogeneous for~$f$}, or \emph{$f$-homogeneous}, if there is a color~$c < k$ such that~$f ([Y]^n) = \{c\}$.
	\end{enumerate}
\end{definition}

\noindent We shall reserve the term \emph{infinite homogeneous set} to mean a set~$H \subseteq \mathbb{Q}$ such that~$H$ is homogeneous for~$f$ and~$(H,\leq_{\mathbb{Q}})$ is order-isomorphic to~$(\omega,\leq_{\mathbb{Q}})$ (or equivalently, to~$(\omega,\leq)$). For~$Y \subseteq \mathbb{Q}$, we say~$Y$ is of \emph{order type~$\eta$} if~$(Y,\leq_{\mathbb{Q}})$ is order-isomorphic to~$(\mathbb{Q},\leq_{\mathbb{Q}})$.

\begin{statement}[Erd\H{o}s-Rado theorem ($\erp$)]
For every coloring~$f : [\Qb]^2\to 2$, there is either an infinite homogeneous set with color~$0$, or a homogeneous set of order type~$\eta$ with color~$1$.
\end{statement}

Note that this statement cannot be strengthened to require the existence of a homogeneous set of order type~$\eta$ (of one or the other color).
Indeed, there exists a coloring~$f : [\Qb]^2 \to 2$ with no homogeneous set of order type~$\eta$ at all (see \cite[Section 1]{FP-TA} for an example).
The reverse mathematics of this Erd\H{o}s-Rado theorem was studied by Frittaion and Patey in~\cite{FP-TA}.
There, it is proved in particular that~$\erp$ does not computably reduce to Ramsey's theorem for pairs and arbitrary many colors
($\erp \not \leq_c \rt^2_{<\infty}$), but the separation of~$\rt^2_2$ from~$\erp$ over~$\RCA_0$ is left open.
Frittaion and Patey notice that the combinatorics of the tree theorem for pairs and the Erd\H{o}s-Rado theorem have a similar flavor,
in that they shar the so-called ``disjoint extension commitment'' (see Section~5 in~\cite{FP-TA}).

However, no formal relation is established between~$\tto^2_2$ and~$\erp$ in reverse mathematics so far.
Intuitively, the Erd\H{o}s-Rado theorem seems to be weaker than the tree theorem for pairs as it requires
only one of the two sides to have homogeneous sets of order type~$\eta$. On the other hand,~$\erp$ is a coloring
of all pairs of rationals, while~$\tto^2_2$ colors only the comparable pairs of nodes.
One can however prove that the stable tree theorem for pairs implies the restriction of~$\erp$
to stable colorings.

\begin{definition}
	Fix~$f : [\Qb]^2 \to k$ and~$X \subseteq \Qb$.
	\begin{enumerate}
		\item~$f$ is \emph{stable over~$X$} if for every~$x \in X$, there is a color~$c < k$ and a finite set~$S \subseteq X$ 
such that~$f(x, y) = c$ for all~$x, y \in X$ such that~$y \not \in S$.
		\item~$f$ is \emph{stable} if it is stable over~$\Qb$.
	\end{enumerate}
\end{definition}

\begin{statement}[$\serp$]
	For every stable coloring~$f : [\Qb]^2 \to k$, there is either an infinite homogeneous set with color~$0$, or a homogeneous of order type~$\eta$ with color~$1$. 
\end{statement}

In this section, we will prove the following theorem,
which answers a question of Frittaion and Patey in~\cite{FP-TA}.

\begin{theorem}\label{thm:rt22-wkl-erp}
$\rca + \wkl \nvdash \rt^2_2 \to \serp$.
\end{theorem}

Before proving Theorem~\ref{thm:rt22-wkl-erp}, we first show that it yields
another proof that the tree Ramsey theorem for pairs is strictly stronger
than~$\rt^2_2$ in reverse mathematics.

\begin{corollary}
$\rca + \wkl \nvdash \rt^2_2 \to \stto^2_2$.
\end{corollary}
\begin{proof}
Thanks to Theorem~\ref{thm:rt22-wkl-erp},
we simply need to prove that~$\rca \vdash \stto^2_2 \to \serp$.
For each~$\sigma \in 2^{<\omega}$, associate a rational~$x_\sigma$ inductively as follows.
First,~$x_\varepsilon = 0$. Having defined~$x_\sigma$, we let~$x_{\sigma 0} = x_\sigma - 2^{-|\sigma|}$
ad~$x_{\sigma 1} = x_\sigma + 2^{-|\sigma|}$. Note that for every set~$T \cong 2^{<\omega}$,
the set~$\{ x_\sigma : \sigma \in T \}$ is of order type~$\eta$.
Given a stable coloring~$f : [\Qb]^2 \to 2$, let
$$
A = \{ \sigma \in 2^{<\omega} : f(x_\sigma, y) = 1 \mbox{ for almost every } y \}
$$
The set~$A$ is~$\Delta^0_2$, so by Shoenfield's limit lemma~\cite{Shoenfield-1959}, 
there is a stable computable function~$g : [2^{<\omega}]^2 \to 2$ such that
$$
A = \{ \sigma : (\exists n)(\forall \tau \succ \sigma)[|\tau| > n \rightarrow g(\sigma, \tau) = 1]\}
$$
Apply~$\stto^2_2$ to get a~$g$-homogeneous set~$T \cong 2^{<\omega}$.
Note that in particular,~$T \subseteq A$ or~$T \subseteq 2^{<\omega} \setminus A$.
The set~$T$ together with~$f$ computes an a set~$H \subseteq \Qb$
such that~$(H, \leq_\Qb)$ is of order type~$\eta$.
This completes the proof.
\end{proof}

The remainder of this section is devoted to the proof of Theorem~\ref{thm:rt22-wkl-erp}.
For this, we will use various combinatorial tools introduced in~\cite{FP-TA}
and generalize their notion of fairness to be preserved by multiple applications of 
Ramsey's theorem for pairs and weak K\"onig's lemma.
We now recall the definitions introduced by Frittaion and Patey in~\cite{FP-TA} and introduce
the generalized notion of fairness.

\begin{definition}
A \emph{simple partition}~$\inter_\Qb(S)$ is a finite sequence of open intervals~$(-\infty, x_0), (x_0, x_1), \dots, (x_{n-1}, +\infty)$
for some set of rationals~$S = \{x_0 <_\Qb \dots <_\Qb x_{n-1}\}$. We set 
$\inter_\Qb(\emptyset) = \{\Qb\}$. 
A simple partition~$I_0, \dots, I_{n-1}$
\emph{refines} another simple partition~$J_0, \dots, J_{m-1}$ if for every~$i < n$,
there is some~$j < m$ such that~$I_i \subseteq J_j$.
Given two simple partitions~$I_0, \dots, I_{n-1}$ and~$J_0, \dots, J_{m-1}$,
the product~$\vec{I} \otimes \vec{J}$ is the simple partition
\[
\{ I \cap J : I \in \vec{I} \wedge J \in \vec{J} \}
\]
\end{definition}

The partition terminology comes from the fact that~$S \cup \bigcup \inter_\Qb(S) = \Qb$
for every finite set of rationals~$S$. In particular,~$\inter_\Qb(S)$ refines~$\inter_\Qb(T)$ if~$T \subseteq S$
and that~$\inter_\Qb(S\cup T)=\inter_\Qb(S)\otimes\inter_\Qb(T)$.

\begin{definition}
An \emph{$m$-by-$n$ matrix}~$M$ is a rectangular array of rationals~$x_{i,j} \in \Qb$
such that~$x_{i,j} <_\Qb x_{i,k}$ for each~$i < m$ and~$j < k < 
n$. 
The~$i$th \emph{row}~$M(i)$ of the matrix~$M$ is the~$n$-tuple of rationals~$x_{i,0} < 
\dots < x_{i,n-1}$. The simple partition~$\inter_\Qb(M)$ is defined by ~$\bigotimes_{i < 
m} \inter_\Qb(M(i))$. In particular,~$\bigotimes_{i <m} \inter_\Qb(M(i))$ refines the 
simple partition~$\inter_\Qb(M(i))$ for each~$i < m$.
\end{definition}

An~\emph{$M$-partition of~$\Qb$} is a simple partition of~$\Qb$ refining~$\bigotimes_{i < m} \inter_\Qb(M(i))$.
Given a simple partition~$\vec{I}$, we want to classify the~$k$-tuples of rationals
according to which interval of~$\vec{I}$ they belong to.
This leads to the notion of~$(\vec{I},k)$-type.

\begin{definition}
Given a simple partition~$I_0, \dots, I_{n-1}$ and some~$k \in \omega$,
an \emph{$(\vec{I},k)$-type} is a tuple~$T_0, \dots, T_{k-1}$
such that~$T_i \in \vec{I}$ for each~$i < k$.
\end{definition}

We now state two simple combinatorial lemmas which are adapted from
Lemma~4.6 and Lemma~4.7 of Frittaion and Patey~\cite{FP-TA}.

\begin{lemma}\label{lem:tuple-has-mtype}
For every simple partition~$I_0, \dots, I_{k-1}$ and for every~$m$-tuple
of rationals~$x_0, \dots, x_{m-1} \in \bigcup_{j < k} I_j$, 
there is an~$(\vec{I},m)$-type~$T_0, \dots, T_{m-1}$
such that~$x_j \in T_j$ for each~$j < m$.
\end{lemma}
\begin{proof}
Fix~$m$ rationals~$x_0, \dots, x_{m-1} \in \bigcup_{i < k} I_i$. For each~$i < m$, 
there is some interval~$T_i \in \vec{I}$ such that~$x_i \in T_i$ since~$x_i \in \bigcup_{j < m} I_j$.
The sequence~$T_0, \dots, T_{m-1}$ is the desired~$(\vec{I}, m)$-type.
\end{proof}

\begin{lemma}\label{lem:mtype-interval-disjoint}
For every~$m$-by-$m$ matrix~$M$, every~$M$-partition~$I_0, \dots, I_{k-1}$ and 
for every~$(I, m)$-type~$T_0, \dots, T_{m-1}$,
there is an~$m$-tuple of intervals~$J_0, \dots, J_{m-1}$ with~$J_i \in \inter_\Qb(M(i))$ such that
\[
(\bigcup_{j < m} T_j) \cap (\bigcup_{i < m} J_i) = \emptyset
\]
\end{lemma}
\begin{proof}
Let~$T_0, \dots, T_{m-1}$ be an~$(\vec{I},m)$-type for some~$M$-partition~$I_0, \dots, I_{k-1}$.
For every~$i < m$, there is some~$J \in \inter_\Qb(M(i))$ such that~$T_i \subseteq J$.
Since~$|\inter_\Qb(M(i))| = m+1$, there is an interval~$J_i \in \inter_\Qb(M(i))$ such that~$(\bigcup_{j <m} T_j) \cap J_i = \emptyset$.
\end{proof}

\begin{definition}
Given an \emph{$m$-by-$n$ matrix}~$M$, an \emph{$M$-formula} 
is a formula~$\varphi$ with distinguished set variables~$U_j$ for each~$j < m$
and~$C_{i,I}$ for each~$i < m$ and~$I \in \inter_\Qb(M(i))$.
An \emph{$M$-valuation~$V$} is a tuple of finite sets~$A_j \subseteq \Qb$ for each~$j < m$
and~$D_{i,I} \subseteq I$ for each~$i < m$ and~$I \in \inter_\Qb(M(i))$.
The~$M$-valuation~$V$ is of type~$\vec{T}$ for some~$M$-partition~$\vec{I}$ and some~$(\vec{I}, m)$-type~$T_0, \dots, T_{m-1}$
if moreover~$A_j \subseteq T_j$ for each~$j < m$.
The valuation~$V$ \emph{satisfies}~$\varphi$ if~$\varphi(A_j : j < m, D_{i,I} : i < m, I \in \inter_\Qb(M(i)))$ holds.
We write~$\varphi(V)$ for~$\varphi(A_j : j < m, D_{i,I} : i < m, I \in \inter_\Qb(M(i)))$.
\end{definition}

Given some valuation~$V = (\vec{A}, \vec{D})$ and some integer~$s$, we write~$V > s$
to say that for every~$x \in (\bigcup \vec{A}) \cup (\bigcup \vec{D})$,~$x > s$. 
Following the terminology of~\cite{LST-2013}, we define 
the notion of essentiality for a formula (an abstract requirement),
which corresponds to the idea that there is room for diagonalization
since the formula is satisfied for arbitrarily far valuations.

\begin{definition}
Given an~$m$-by-$n$ matrix~$M$, an~$M$-partition~$\vec{I}$,
and an~$(\vec{I}, m)$-type~$\vec{T}$, we say that an~$M$-formula~$\varphi$
is \emph{$\vec{T}$-essential} if for every~$s \in \omega$,
there is an~$M$-valuation~$V > s$ of type~$\vec{T}$
such that~$\varphi(V)$ holds.
\end{definition}

We simply say that~$\varphi$ is \emph{essential} if it is~$\vec{T}$-essential
for some ~$M$-partition~$\vec{I}$ and some~$(\vec{I}, m)$-type~$\vec{T}$.
The notion of~$\er$-fairness is defined accordingly. If some formula
is essential, that is, gives enough room for diagonalization, then there is
an actual valuation which will diagonalize against the~$\erps$-instance.

\begin{definition}
Fix two sets~$A_0, A_1 \subseteq \Qb$.
Given an~$m$-by-$n$ matrix~$M$, an~$M$-valuation~$V = (\vec{B}, \vec{D})$ 
\emph{diagonalizes} against~$A_0, A_1$
if~$\bigcup \vec{B} \subseteq A_1$ and for every~$i < m$, there is some~$I \in \inter_\Qb(M(i))$ 
such that~$I \subseteq A_0$.
A set~$X$ is \emph{$n$-$\er$-fair} for~$A_0, A_1$ if for every~$m$, every ~$m$-by-$2^nm$ matrix~$M$,
and every~$\Sigma^{0,X}_1$ essential~$M$-formula, there is an~$M$-valuation~$V$ diagonalizing against~$A_0, A_1$ 
such that~$\varphi(V)$ holds.
A set~$X$ is \emph{$\er$-fair} for~$A_0, A_1$ if it is~$n$-$\er$-fair for~$A_0, A_1$ for some~$n \geq 0$.
\end{definition}

Of course, if~$Y \leq_T X$, then every~$\Sigma^{0,Y}_1$ formula is~$\Sigma^{0,X}_1$.
As an immediate consequence, if~$X$ is~$n$-$\er$-fair for some~$A_0, A_1$ and~$Y \leq_T X$, then~$Y$ is~$n$-$\er$-fair for~$A_0, A_1$.

\begin{definition}
Fix a~$\Pi^1_2$ statement~$\Psf$.
	\begin{enumerate}
	\item~$\Psf$ admits \emph{$\er$-fairness preservation} (respectively, \emph{$n$-$\er$-fairness preservation}) if for all sets~$A_0, A_1 \subseteq \Qb$,
every set~$C$ which is~$\er$-fair (respectively,~$n$-$\er$-fair) for~$A_0, A_1$ and every~$C$-computable~$\Psf$-instance~$X$,
there is a solution~$Y$ to~$X$ such that~$Y \oplus C$ is~$\er$-fair (respectively,~$n$-$\er$-fair) for~$A_0, A_1$.
	\item~$\Psf$ admits strong \emph{$\er$-fairness preservation} (respectively, \emph{$n$-$\er$-fairness preservation}) if for all sets~$A_0, A_1 \subseteq \Qb$,
every set~$C$ which is~$\er$-fair (respectively,~$n$-$\er$-fair) for~$A_0, A_1$ and every~$\Psf$-instance~$X$,
there is a solution~$Y$ to~$X$ such that~$Y \oplus C$ is~$\er$-fair (respectively,~$n$-$\er$-fair) for~$A_0, A_1$.
\end{enumerate}
\end{definition}

We create a non-effective instance of~$\erps$ which will serve as a bootstrap for~$\er$-fairness preservation.
The proof is very similar to Lemma 4.11 in~\cite{FP-TA}.

\begin{lemma}\label{lem:partition-emptyset-er-fair}
There exists a~$\Delta^0_2$ partition~$A_0 \cup A_1 = \Qb$ such that
$\emptyset$ is 0-$\er$-fair for~$A_0, A_1$.
\end{lemma}
\begin{proof}
The proof is done by a no-injury priority construction.
Interpret each~$s \in \omega$ as a tuple~$\langle M, \varphi, \vec{I}, \vec{T} \rangle$ where
$M$ is an~$m$-by-$m$ matrix,~$\varphi$ is a~$\Sigma^0_1$~$M$-formula,
$\vec{I}$ is an~$M$-partition, and~$\vec{T}$ is an~$(\vec{I}, m)$-type.
We want to satisfy the following requirements for each~$s = \langle M, \varphi, \vec{I}, \vec{T} \rangle$.

\begin{quote}
$\Rcal_s$: If~$\varphi$ is~$\vec{T}$-essential, then~$\varphi(V)$ holds
for some~$M$-valuation~$V$ diagonalizing against~$A_0, A_1$.
\end{quote}

The requirements are given a standard priority ordering.
The sets~$A_0$ and~$A_1$ are constructed by a~$\emptyset'$-computable list of 
finite approximations~$A_{i,0} \subseteq A_{i,1} \subseteq \dots$
such that all elements added to~$A_{i,s+1}$ from~$A_{i,s}$
are strictly greater than the maximum of~$A_{i,s}$ (in the~$\Nb$ order) for each~$i < 2$. 
We then let~$A_i = \bigcup_s A_{i,s}$ which will be a~$\Delta^0_2$ set.
At stage 0, set~$A_{0,0} = A_{1,0} = \emptyset$. Suppose that at stage~$s$,
we have defined two disjoint finite sets~$A_{0,s}$ and~$A_{1,s}$ such that
\begin{itemize}
	\item[(i)]~$A_{0,s} \cup A_{1,s} = [0,b]_\Nb$ for some integer~$b \geq s$
	\item[(ii)]~$\Rcal_{s'}$ is satisfied for every~$s' < s$
\end{itemize}
Let~$\Rcal_s$ be the requirement such that~$s = \langle M, \varphi, \vec{I}, \vec{T} \rangle$.
Decide~$\emptyset'$-computably whether there is some~$M$-valuation~$V > b$ of type~$\vec{T}$
such that~$\varphi(V)$ holds. If so, effectively fetch~$\vec{T} = T_0, \dots, T_{m-1}$
and such a~$V > b$. Let~$d$ be an upper bound (in the~$\Nb$ order) on the rationals in~$V$.
By Lemma~\ref{lem:mtype-interval-disjoint}, for each~$i < m$,
there is some~$J_i \in \inter_\Qb(M(i))$ such that
\[
(\bigcup_{j < m} T_j) \cap (\bigcup_{i < m} J_i) = \emptyset
\]
Set~$A_{0,s+1} = A_{0,s} \bigcup_{i < m} J_i \cap (b,d]_\Nb$ and~$A_{1,s+1} = [0,d]_\Nb \setminus A_{0,s+1}$.
This way,~$A_{0,s+1} \cup A_{1,s+1} = [0, d]_\Nb$.
By the previous equation,~$\bigcup_{j < m} T_j \cap (b, d]_\Nb \subseteq [0,d]_\Nb \setminus A_{0,s+1}$ 
and the requirement~$\Rcal_s$ is satisfied.
If no such~$M$-valuation is found, the requirement~$\Rcal_s$ is vacuously satisfied.
Set~$A_{0,s+1} = A_{0,s} \cup \{b\}$ and~$A_{1,s+1} = A_{1,s}$.
This way,~$A_{0,s+1} \cup A_{1,s+1} = [0, b+1]_\Nb$.
In any case, go to the next stage. This completes the construction.
\end{proof}

\begin{theorem}\label{thm:er22-not-er-fairness}
$\erp$ does not admit~$\er$-fairness preservation.
\end{theorem}
\begin{proof}
Let~$A_0 \cup A_1 = \Qb$ be the~$\Delta^0_2$ partition constructed in Lemma~\ref{lem:partition-emptyset-er-fair}.
By Shoenfield's limit lemma~\cite{Shoenfield-1959}, there is a computable function~$f : [\Qb]^2 \to 2$ such 
that for each~$x \in \Qb$,~$\lim_s f(x, s)$ exists and~$x \in A_{\lim_s h(x, s)}$.
Let~$D \subseteq \Qb$ be an~$\erp$-solution to~$f$.

Fix any~$n \geq 0$. 
We claim that~$D$ is not~$n$-$\er$-fair for~$A_0, A_1$.
Let~$M = (x_0 <_\Qb \dots <_\Qb x_{2^n-1})$ be the~$1$-by-$2^n$ matrix composed
of the~$2^n$ first elements of~$D$ is the natural order. We have two cases.

\begin{itemize}
	\item Case 1:~$D$ is an infinite 0-homogeneous set.
	Let~$\varphi(U, C_I : I \in \inter_\Qb(M))$ be the~$\Sigma^{0,D}_1$~$M$-formula
	which holds if~$U$ is a non-empty subset of~$D$.
	We claim that~$\varphi$ is essential.
	Since~$\Qb \setminus \bigcup \inter_\Qb(M)$ is finite, then for every~$s$,
  there is some~$y \in D \cap \bigcup \inter_\Qb(M)$ such that~$y >_\Nb s$.
	By Lemma~\ref{lem:tuple-has-mtype}, there is a~$(\inter_\Qb(M),1)$-type~$T$ such that~$y \in T$. 
	The singleton~$\{y\}$ forms an~$M$-valuation~$V > s$ of type~$T$
	such that~$\varphi(V)$ holds. Therefore, by the pigeonhole principle,
	the~$M$-formula~$\varphi$ is~$T$-essential for some~$(\inter_\Qb(M),1)$-type~$T$.
	For every~$M$-valuation~$V = (B, D_I : I \in \inter_\Qb(M))$ such that~$\varphi(V)$ holds,
	it cannot be the case that~$B \subseteq A_1$ since it would contradict
	the fact that~$B$ is a non-empty subset of~$D \subseteq A_0$.

	\item Case 2:~$D$ is a dense 1-homogeneous set.
	Let~$\varphi(U, C_{0, I} : I \in \inter_\Qb(M))$ be the~$\Sigma^{0,D}_1$~$M$-formula
	which holds if for each~$I \in \inter_\Qb(M)$,~$C_I$ is a non-empty subset of~$D \cap I$.
	We claim that~$\varphi$ is essential.
	Since~$D$ is dense, there is some collection~$(y_I \in D \cap I : I \in \inter_\Qb(M))$
	such that~$y_I >_\Nb s$. Taking~$D = \emptyset$, the~$y$'s form an~$M$-valuation~$V > s$ of every~$(\inter_\Qb(M), 1)$-type
	such that~$\varphi(V)$ holds. Therefore, the~$M$-formula~$\varphi$ is~$T$-essential for every~$(\inter_\Qb(M), 1)$-type~$T$. 
	For every~$M$-valuation~$V = (B, D_I : I \in \inter_\Qb(M))$ such that~$\varphi(V)$ holds,
	there is no~$I \in \inter_\Qb(M)$ such that~$D_I \subseteq A_0$ since it would contradict
	the fact that~$D_I$ is a non-empty subset of~$D \subseteq A_1$.
\end{itemize}
In both cases,~$\varphi$ is essential, but has no~$M$-valuation
which diagonalizes against~$A_0, A_1$, so~$S$ is not~$n$-$\er$-fair for~$A_0, A_1$.
\end{proof}

We now prove~$\er$-fairness preservation for various principles in reverse mathematics,
namely, weak K\"onig's lemma, cohesiveness and~$\rt^2_2$. We prove independently that
they admit~$\er$-fairness preservation, and then use the compositional
nature of the notion of preservation to deduce that the conjunction of these
principles do not imply~$\erp$ over~$\rca$.

\begin{theorem}\label{thm:wkl-n-er-fairness}
For every~$n \geq 0$,~$\wkl$ admits~$n$-$\er$-fairness preservation.
\end{theorem}
\begin{proof}
Let~$C$ be a set~$n$-$\er$-fair for some sets~$A_0, A_1 \subseteq \Qb$,
and let~$\Tcal \subseteq 2^{<\omega}$ be a~$C$-computable infinite binary tree.
We construct an infinite decreasing sequence of~$C$-computable subtrees~$\Tcal = \Tcal_0 \supseteq \Tcal_1 \supseteq \dots$
such that for every path~$P$ through~$\bigcap_s \Tcal_s$,~$P \oplus C$ is~$n$-$\er$-fair for~$A_0, A_1$.
Note that the intersection~$\bigcap_s \Tcal_s$ is non-empty since the~$\Tcal$'s are infinite trees.
More precisely, if we interpret~$s$ as a tuple~$\langle m, M, \vec{I}, \vec{T}, \varphi \rangle$ where 
$M$ is an~$m$-by-$2^nm$ matrix,~$\vec{I}$ is an~$M$-partition,~$\vec{T}$ is an~$(\vec{I}, m)$-type, and~$\varphi(G,U)$
is a~$\Sigma^{0,C}_1$~$M$-formula, we want to satisfy the following requirement.

\begin{quote}
$\Rcal_s$ : For every path~$P$ through~$\Tcal_{s+1}$, either~$\varphi(P, U)$ is not~$\vec{T}$-essential,
or~$\varphi(P, V)$ holds for some~$M$-valuation~$V$ diagonalizing against~$A_0, A_1$.
\end{quote}

Given two~$M$-valuations~$V_0 = (\vec{R}, \vec{S})$ and~$V_1 = (\vec{D}, \vec{E})$, 
we write~$V_0 \subseteq V_1$
to denote the pointwise subset relation, that is, 
$R_j \subseteq D_j$ and~$S_{i,I} \subseteq E_{i,I}$
for every~$i < m$,~$j < n$ and~$I \in \inter_\Qb(M(i))$.
At stage~$s = \langle m, M, \vec{I}, \vec{T}, \varphi \rangle$, given some infinite, computable binary tree~$\Tcal_s$, 
define the~$m$-by-$n$~$\Sigma^{0,C}_1$~$M$-formula
\[
\psi(U) = (\exists n)(\forall \tau \in T_s \cap 2^n)(\exists \tilde{V} \subseteq U)\varphi(\tau, \tilde{V})
\]
We have two cases.
In the first case,~$\psi(U)$ is not~$\vec{T}$-essential with some witness~$t$. By compactness,
the following set is an infinite~$C$-computable subtree of~$\Tcal_s$:
\[
\Tcal_{s+1} = \{ \tau \in \Tcal_s : (\mbox{for every } M\mbox{-valuation } V > t 
	\mbox{ of type } \vec{T})\neg \varphi(\tau, V) \}
\]
The tree~$\Tcal_{s+1}$ is defined so that~$\varphi(P, U)$
is not~$\vec{T}$-essential for every~$P \in [\Tcal_{s+1}]$.

In the second case,~$\psi(U)$ is~$\vec{T}$-essential. By~$n$-$\er$-fairness of~$C$ for~$A_0, A_1$,
there is an~$M$-valuation~$V = (\vec{R}, \vec{S})$ diagonalizing against~$A_0, A_1$ such that~$\psi(\vec{R}, \vec{S})$ holds.
We claim that for every path~$P \in [\Tcal_s]$,
$\varphi(P, \tilde{V})$ holds for some~$M$-valuation~$\tilde{V}$ diagonalizing against~$A_0, A_1$.
Fix some path~$P \in [\Tcal_s]$. Unfolding the definition of~$\psi(V)$, 
there is some~$u$ such that~$\varphi(P \uh u, \tilde{V})$ holds
for some~$M$-valuation~$\tilde{V} = (\tilde{D}, \vec{E}) \subseteq V = (\vec{R}, \vec{S})$. 
By definition of~$V$ diagonalizing against~$A_0, A_1$, 
$\bigcup \vec{R} \subseteq A_1$ and for every~$i < m$, there is some~$I \in \inter_\Qb(M(i))$ 
such that~$S_{i,I} \subseteq A_0$. In particular,~$\bigcup \vec{D} \subseteq \bigcup \vec{R} \subseteq A_1$
and for every~$i < m$, there is some~$I \in \inter_\Qb(M(i))$  such that~$E_{i,I} \subseteq S_{i,I} \subseteq A_0$.
Therefore,~$\tilde{V}$ diagonalizes against~$A_0, A_1$. This completes the claim.
Take~$\Tcal_{s+1} = \Tcal_s$ and go to the next stage.
This completes the proof of Theorem~\ref{thm:wkl-n-er-fairness}.
\end{proof}

As previously noted, preserving~$n$-$\er$-fairness for every~$n$ implies preserving~$\er$-fairness.
However, we really need the fact that~$\wkl$ admits~$n$-$\er$-fairness preservation
and not only~$\er$-fairness preservation in the proof of Theorem~\ref{thm:rt12-strong-er-fairness}.

\begin{corollary}\label{cor:wkl-er-fairness-preservation}
$\wkl$ admits~$\er$-fairness preservation.
\end{corollary}




Recall Proposition \ref{prop:CJS_decomp}, that~$\RT^2_2$ is equivalent over~$\RCA_0$ to~$\COH + \SRT^2_2$. A more natural way to think of this is as a decomposition of~$\rt^2_2$
into~$\COH$ and a non-effective instances of~$\rt^1_2$. 
Indeed, given an effective instance~$f : [\omega]^2 \to 2$ of~$\rt^2_2$,
$\coh$ implies the existence of an infinite set~$H$ such that~$f : [H]^2 \to 2$ is stable.
By Schoenfield's limit lemma~\cite{Shoenfield-1959},
the stable coloring~$f : [H]^2 \to 2$ can be seen as the~$\Delta^0_2$ approximation of a~$\emptyset'$-computable
instance~$\tilde{f} : H \to 2$ of~$\rt^1_2$. Moreover, we can~$H$-compute an infinite~$f$-homogeneous set 
from any~$\tilde{f}$-homogeneous set.
We shall therefore prove independently~$\er$-fairness preservation of~$\coh$
and strong~$\er$-fairness preservation of~$\rt^1_2$ to deduce that~$\rt^2_2$ admits~$\er$-fairness preservation.

\begin{theorem}\label{thm:coh-n-er-fairness}
For every~$n \geq 0$,~$\coh$ admits~$n$-$\er$-fairness preservation.
\end{theorem}
\begin{proof}
Let~$C$ be a set~$n$-$\er$-fair for some sets~$A_0, A_1 \subseteq \Qb$,
and let~$R_0, R_1, \dots$ be a~$C$-computable sequence of sets.
We will construct an~$\vec{R}$-cohesive set~$G$ such that 
$G \oplus C$ is~$n$-$\er$-fair for~$A_0, A_1$.
The construction is done by a Mathias forcing, whose conditions are pairs~$(F, X)$
where~$X$ is a~$C$-computable set. The result is a direct consequence of the following lemma.

\begin{lemma}\label{lem:coh-preservation-lemma}
For every condition~$(F, X)$, every~$m$-by-$2^nm$ matrix~$M$
and every~$\Sigma^{0,C}_1$~$M$-formula~$\varphi(G, U)$, 
there exists an extension~$d = (E, Y)$ such that
$\varphi(G, U)$ is not essential for every set~$G$ satisfying~$d$, or~$\varphi(E, V)$ holds
for some~$M$-valuation~$V$ diagonalizing against~$A_0, A_1$.
\end{lemma}
\begin{proof}
Define the~$\Sigma^{0,C}_1$~$M$-formula
$\psi(U) = (\exists G \supseteq F)[G \subseteq F \cup X \wedge \varphi(G, U)]$.
By~$n$-$\er$-fairness of~$C$ for~$A_0, A_1$, 
either~$\psi(U)$ is not essential, or~$\psi(V)$ holds for some~$M$-valuation~$V$
diagonalizing against~$A_0, A_1$.
In the former case, the condition~$(F,X)$ already satisfies the desired property with the same witnesses.
In the latter case, by the finite use property, there exists a finite set~$E$ satisfying~$(F, X)$ such that~$\varphi(E, V)$ holds.
Let~$Y = X \setminus [0, max(E)]$. The condition~$(E, Y)$ is a valid extension.
\end{proof}

Let~$\Fcal = \{c_0, c_1, \dots \}$ be a sufficiently generic filter,
where~$c_s = (F_s, X_s)$, and let~$G = \bigcup_s F_s$.
In particular,~$G$ is infinite and~$\vec{R}$-cohesive.
By Lemma~\ref{lem:coh-preservation-lemma},~$G \oplus C$ is~$n$-$\er$-fair for~$A_0, A_1$.
This completes the proof of Theorem~\ref{thm:coh-n-er-fairness}.
\end{proof}

\begin{corollary}
$\coh$ admits~$\er$-fairness preservation.
\end{corollary}

The next theorem is the reason why we use the notion of
$\er$-fairness instead of~$n$-$\er$-fairness
in our separation of~$\rt^2_2$ from~$\erp$.
Indeed, given an instance of~$\rt^1_2$ and a set~$C$
which is~$n$-$\er$-fair for some sets~$A_0, A_1$, the proof
constructs a solution~$H$ such that~$H \oplus C$ is~$(n+1)$-$\er$-fair for~$A_0, A_1$.

\begin{theorem}\label{thm:rt12-strong-er-fairness}
$\rt^1_2$ admits strong~$\er$-fairness preservation.
\end{theorem}
\begin{proof}
Let~$C$ be a set~$n$-$\er$-fair for some sets~$A_0, A_1 \subseteq \Qb$,
and let~$B_0 \cup B_1 = \omega$ be a (non-necessarily effective) 2-partition of~$\omega$.
Suppose that there is no infinite set~$H \subseteq B_0$ or~$H \subseteq B_1$
such that~$H \oplus C$ is~$n$-$\er$-fair for~$A_0, A_1$, since otherwise we are done.
We construct a set~$G$ such that both~$G \cap B_0$ and~$G \cap B_1$ are infinite.
We need therefore to satisfy the following requirements for each~$p \in \omega$.
\[
  \Ncal_p : \hspace{20pt} (\exists q_0 > p)[q_0 \in G \cap B_0] 
		\hspace{20pt} \wedge \hspace{20pt} (\exists q_1 > p)[q_1 \in G \cap B_1] 
\]
Furthermore, we want to ensure that one of~$(G \cap B_0) \oplus C$ 
and~$(G \cap B_1) \oplus C$ is~$\er$-fair for~$A_0, A_1$. To do this, we will satisfy the following requirements
for every integer~$m$, every~$m$-by-$2^{n+1}m$ matrices~$M_0$ and~$M_1$,
every~$\Sigma^{0,C}_1$~$M_0$-formula~$\varphi_0(H, U)$ and~$M_1$-formula~$\varphi_1(H, U)$.
\[
  \Qcal_{\varphi_0, M_0, \varphi_1, M_1} : \hspace{20pt} 
		\Rcal_{\varphi_0,M_0}^{G \cap B_0} \hspace{20pt} \vee \hspace{20pt}  \Rcal_{\varphi_1,M_1}^{G \cap B_1}
\]
where~$\Rcal_{\varphi, M}^H$ holds if~$\varphi(H, U)$ is not essential
or~$\varphi(H, V)$ holds for some~$M$-valuation~$V$ diagonalizing against~$A_0, A_1$.
We first justify that if every~$\Qcal$-requirement is satisfied, then either~$(G \cap B_0) \oplus C$
or~$(G \cap B_1) \oplus C$ is~$(n+1)$-$\er$-fair for~$A_0, A_1$.
By the usual pairing argument, for every~$m$, there is some side~$i < 2$ such that
the following property holds:
\begin{quote}
(P) For every~$m$-by-$2^{n+1}m$ matrix~$M$ and every~$\Sigma^{0,C}_1$~$M$-formula~$\varphi(G \cap B_i, U)$,
either~$\varphi(G \cap B_i, U)$ is not essential, or~$\varphi(G \cap B_i, V)$ holds for some~$M$-valuation~$V$
diagonalizing against~$A_0, A_1$.
\end{quote}
By the infinite pigeonhole principle, 
there is some side~$i < 2$ such that (P) holds for infinitely many~$m$.
By a cropping argument, if (P) holds for~$m$ and~$q < m$, then (P) holds for~$q$.
Therefore (P) holds for every~$m$ on side~$i$. In other words,~$(G \cap B_i) \oplus C$
is~$(n+1)$-$\er$-fair for~$A_0, A_1$.

We construct our set~$G$ by forcing. Our conditions are Mathias conditions~$(F, X)$,
such that~$X \oplus C$ is~$n$-$\er$-fair for~$A_0, A_1$.
We now prove the progress lemma, stating that we can force both~$G \cap B_0$
and~$G \cap B_1$ to be infinite.

\begin{lemma}\label{lem:rt12-er-fairness-progress}
For every condition~$c = (F, X)$, every~$i < 2$ and every~$p \in \omega$
there is some extension~$d = (E, Y)$ such that~$E \cap B_i \cap (p,+\infty) \neq \emptyset$.
\end{lemma}
\begin{proof}
Fix~$c$,~$i$ and~$p$. If~$X \cap B_i \cap (p,+\infty) = \emptyset$,
then~$X \cap (p,+\infty)$ is an infinite subset of~$B_{1-i}$.
Moreover,~$X \cap (p,+\infty)$ is~$n$-$\er$-fair for~$A_0, A_1$, contradicting our hypothesis.
Thus, there is some~$q > p$ such that~$q \in X \cap B_i \cap (p,+\infty)$.
Take~$d = (F \cup \{q\}, X \setminus [0,q])$ as the desired extension.
\end{proof}

Given two~$m$-by-$n$ matrices~$M_0$ and~$M_1$, we denote by~$M_0 \sqcup M_1$
the~$2m$-by-$n$ matrix obtained by putting the adding the rows of~$M_1$ below the rows of~$M_0$.
Note that every~$M_0 \sqcup M_1$-partition is both an~$M_0$-partition
and an~$M_1$-partition. An~$M_0 \sqcup M_1$-valuation~$V$ can be written as~$(V_0, V_1)$,
where~$V_0$ is an~$M_0$-valuation and~$V_1$ is an~$M_1$-valuation.
Note that if~$V$ diagonalizes against~$A_0, A_1$, then so do both~$V_0$ and~$V_1$.
We now prove the core lemma stating that we can satisfy each~$\Qcal$-requirement.
A condition~$c$ \emph{forces} a requirement~$\Qcal$
if~$\Qcal$ is holds for every set~$G$ satisfying~$c$.

\begin{lemma}\label{lem:rt12-er-fairness-forcing}
For every condition~$c = (F, X)$, every integer~$m$, every~$m$-by-$2^{n+1}m$ matrices~$M_0$ and~$M_1$,
and for every~$\Sigma^{0,C}_1$~$M_0$-formula~$\varphi_0(H, U)$ and~$M_1$-formula~$\varphi_1(H, U)$,
there is an extension~$d = (E, Y)$ forcing~$\Qcal_{\varphi_0, M_0, \varphi_1, M_1}$.
\end{lemma}
\begin{proof}
Let~$\psi(U_0,U_1)$ be the~$\Sigma^{0,X \oplus C}_1$~$M_1 \sqcup M_2$-formula which holds
if for every 2-partition~$Z_0 \cup Z_1 = X$, there is some~$i < 2$,
some finite set~$E \subseteq Z_i$
and an~$M_i$-valuation~$V \subseteq U_i$ such that~$\varphi_i((F \cap B_i) \cup E, V)$
holds. By~$n$-$\er$-fairness of~$X \oplus C$, we have two cases.

In the first case,~$\psi(U_0, U_1)$ is not essential, with some witness threshold~$t \in \omega$
and witness~$M_0 \sqcup M_1$-partition~$\vec{J}$.
By compactness, for every~$(\vec{J}, 2m)$-type~$\vec{T} = \vec{T}^0, \vec{T}^1$,
the~$\Pi^{0, X \oplus C}_1$ class~$\Ccal_{\vec{T}}$ of sets~$Z_0 \oplus Z_1$
such that~$Z_0 \cup Z_1 = \omega$ and for every~$i < 2$
and every finite set~$E \subseteq Z_i$, there is no~$M_i$-valuation~$V > t$ of type~$\vec{T}^i$
such that~$\varphi_i((F \cap B_i) \cup E, V)$ holds is non-empty.
By~$n$-$\er$-fairness preservation of~$\wkl$ (Theorem~\ref{thm:wkl-n-er-fairness}), 
for every~$(\vec{J}, 2m)$-type~$\vec{T}$, there is a 2-partition
$Z^{\vec{T}}_0 \oplus Z^{\vec{T}}_1 \in \Ccal_{\vec{T}}$ such that~
$\bigoplus_{\vec{T}} Z^{\vec{T}}_0 \oplus Z^{\vec{T}}_1 \oplus C$ is~$n$-$\er$-fair for~$A_0, A_1$.
Let~$Z \subseteq X$ be an infinite set~$n$-$\er$-fair for~$A_0, A_1$
such that~$Z \subseteq Z^{\vec{T}}_0$ or~$Z \subseteq Z^{\vec{T}}_1$ for each~$(\vec{J}, 2m)$-type~$\vec{T}$.
Since the~$(\vec{J},2m)$-types are exactly the pairs of all~$(\vec{J}, m)$-types, by the usual pairing argument, 
there is one side~$i < 2$ such that~$(F, Z)$ forces~$\varphi_i(H, U)$ not to be essential.
The condition~$d = (F, Z)$ is an extension forcing~$\Qcal_{\varphi_0, M_0, \varphi_1, M_1}$
by the~$i$th side.

In the second case,~$\psi(V_0, V_1)$ holds for some~$M_0 \sqcup M_1$-valuation~$(V_0, V_1)$
diagonalizing against~$A_0, A_1$. Let~$Z_0 = X \cap B_0$ and~$Z_1 = X \cap B_1$.
By hypothesis, there is some~$i < 2$, some finite set~$E \subseteq Z_i = X \cap B_i$
and some~$M_i$-valuation~$V \subseteq V_i$ such that~$\varphi_i((F \cap B_i) \cup E, V)$ holds.
Since~$V \subseteq V_i$, the~$M_i$-valuation~$V$ diagonalizes against~$A_0, A_1$.
The condition~$d = (F \cup E, X \setminus [0, max(E)])$ is an extension
forcing~$\Qcal_{\varphi_0, M_0, \varphi_1, M_1}$ by the~$i$th side.
\end{proof}

Let~$\Fcal = \{c_0, c_1, \dots \}$ be a sufficiently generic filter, where~$c_s = (F_s, X_s)$, and let~$G = \bigcup_s F_s$.
By Lemma~\ref{lem:rt12-er-fairness-progress},~$G \cap B_0$ and~$G \cap B_1$ are both infinite.
By Lemma~\ref{lem:rt12-er-fairness-forcing}, one of~$G \cap B_0$ and~$G \cap B_1$ is~$\er$-fair for~$A_0, A_1$.
This completes the proof of Theorem~\ref{thm:rt12-strong-er-fairness}.
\end{proof}

\begin{theorem}\label{thm:rt22-er-fairness-preservation}
$\rt^2_2$ admits~$\er$-fairness preservation.
\end{theorem}
\begin{proof}
Fix any set~$C$~$\er$-fair for some sets~$A_0, A_1 \subseteq \Qb$ and any~$C$-computable
coloring~$f : [\omega]^2 \to 2$.
Consider the uniformly~$C$-computable sequence of sets~$\vec{R}$ defined for each~$x \in \omega$ by
\[
R_x = \{s \in \omega : f(x,s) = 1\}
\]
As~$\coh$ admits~$\er$-fairness preservation, there is
some~$\vec{R}$-cohesive set~$G$ such that~$G \oplus C$ is~$\er$-fair for~$A_0, A_1$.
The set~$G$ induces a~$(G \oplus C)'$-computable coloring~$\tilde{f} : \omega \to 2$ defined by:
\[
(\forall x \in \omega) \tilde{f}(x) = \lim_{s \in G} f(x,s)
\]
As~$\rt^1_2$ admits strong~$\er$-fairness preservation,
there is an infinite~$\tilde{f}$-homogeneous set~$H$ such that
$H \oplus G \oplus C$ is~$\er$-fair for~$A_0, A_1$.
The set~$H \oplus G \oplus C$ computes an infinite~$f$-homogeneous set.
\end{proof}

We are now ready to prove Theorem~\ref{thm:rt22-wkl-erp}.

\begin{proof}[Proof of Theorem~\ref{thm:rt22-wkl-erp}]
By Theorem~\ref{thm:rt22-er-fairness-preservation}
and Corollary~\ref{cor:wkl-er-fairness-preservation},~$\rt^2_2$ and~$\wkl$ admit~$\er$-fairness preservation.
By Theorem~\ref{thm:er22-not-er-fairness},~$\erp$ does not admit~$\er$-fairness preservation.
We conclude by Lemma~3.4.2 in~\cite{Patey2016reverse}.
\end{proof}

\section{Questions}
We conclude by listing several questions concerning the above principles left open by our work. The first, and in some sense most pressing, asks to clarify the precise relationship between $\TT^2_2$ and $\erp$. This also appears as Question 5.2 in \cite{FP-TA}.

\begin{question}
	Over $\RCA_0$, does $\TT^2_2$ imply $\erp$ or conversely?	
\end{question}

\noindent In a similar spirit, we can also ask what the relationships of these principles to some other well-known subsystems. First, we can ask whether an analogue of our main theorem holds for $\erp$ in place of $\TT^2_2$.

\begin{question}
	Over $\RCA_0$, does $\erp$ imply $\ACA_0$?	
\end{question}

\noindent Also, as we have shown that $\TT^2_2$ does not imply $\ACA_0$, the following question is a natural follow-up.

\begin{question}
	Over $\RCA_0$, does $\TT^2_2$ imply $\WKL_0$?	
\end{question}

\noindent We conjecture the answer to be no.

Another question concerns the relationship between the stable and cohesive forms of $\TT^2_2$.

\begin{question}\
	\begin{enumerate}
		\item Does $\STT^2_2$ imply $\CTT^2_2$ over $\RCA_0$, or at least over $\omega$ models?
		\item Does $\STT^2_2$ imply $\COH$ over $\RCA_0$, or at least over $\omega$ models?
	\end{enumerate}
\end{question}

\noindent By itself, these questions are technical and not particularly natural. But even partial results here would likely shed light on the corresponding question for linear orders, i.e., the question of whether $\SRT^2_2$ implies $\COH$ in $\omega$-models. The latter is a longstanding and major open problem. (See, e.g., Dzhafarov et al. \cite[Section 1]{DPSW-TA} for a discussion.) Of course, we again conjecture the answer to be no, to both questions. More generally, it would be interesting and potentially insightful to see which computability-theoretic and reverse mathematical questions surrounding $\SRT^2_2$ are simpler to answer for $\STT^2_2$.


\bibliographystyle{plain}
\bibliography{Papers}

\begin{thebibliography}{10}

\bibitem{Brattka-bib}
Vasco Brattka.
\newblock Bibliography on {W}eihrauch complexity, website: {\tt
  http://cca-net.de/publications/weibib.php}.

\bibitem{CDHS-2014}
Peter~A. Cholak, Damir~D. Dzhafarov, Jeffry~L. Hirst, and Theodore~A. Slaman.
\newblock Generics for computable {M}athias forcing.
\newblock {\em Ann. Pure Appl. Logic}, 165(9):1418--1428, 2014.

\bibitem{CJS-2001}
Peter~A. Cholak, Carl~G. Jockusch, and Theodore~A. Slaman.
\newblock On the strength of {R}amsey's theorem for pairs.
\newblock {\em J. Symbolic Logic}, 66(1):1--55, 2001.

\bibitem{CLY-2014}
C.~T. Chong, Wei Li, and Yue Yang.
\newblock Nonstandard models in recursion theory and reverse mathematics.
\newblock {\em Bull. Symb. Log.}, 20(2):170--200, 2014.

\bibitem{CHM-2005}
Jennifer Chubb, Jeffry~L. Hirst, and Timothy~H. McNicholl.
\newblock Reverse mathematics and partitions of trees.
\newblock 2005.

\bibitem{CG-2016}
Jared Corduan and Marcia Groszek.
\newblock Reverse mathematics and {R}amsey properties of partial orderings.
\newblock {\em Notre Dame J. Form. Log.}, 57(1):1--25, 2016.

\bibitem{CGM-2010}
Jared Corduan, Marcia~J. Groszek, and Joseph~R. Mileti.
\newblock Reverse mathematics and {R}amsey's property for trees.
\newblock {\em J. Symbolic Logic}, 75(3):945--954, 2010.

\bibitem{DDHMS-2016}
Fran{\c{c}}ois~G. Dorais, Damir~D. Dzhafarov, Jeffry~L. Hirst, Joseph~R.
  Mileti, and Paul Shafer.
\newblock On uniform relationships between combinatorial problems.
\newblock {\em Trans. Amer. Math. Soc.}, 368(2):1321--1359, 2016.

\bibitem{DH-2010}
Rodney~G. Downey and Denis~R. Hirschfeldt.
\newblock {\em Algorithmic randomness and complexity}.
\newblock Theory and Applications of Computability. Springer, New York, 2010.

\bibitem{Dzhafarov-zoo}
Damir~D. Dzhafarov.
\newblock The {RM} {Z}oo, website: {\tt http://rmzoo.uconn.edu}, 2015.

\bibitem{DHL-2010}
Damir~D. Dzhafarov, Jeffry~L. Hirst, and Tamara~J. Lakins.
\newblock {R}amsey's theorem for trees: the polarized tree theorem and notions
  of stability.
\newblock {\em Arch. Math. Logic}, 49(3):399--415, 2010.

\bibitem{DJ-2009}
Damir~D. Dzhafarov and Carl~G. Jockusch, Jr.
\newblock {R}amsey's theorem and cone avoidance.
\newblock {\em J. Symbolic Logic}, 74(2):557--578, 2009.

\bibitem{DPSW-TA}
Damir~D. Dzhafarov, Ludovic Patey, Reed Solomon, and Linda~Brown Westrick.
\newblock {R}amsey's theorem for singletons and strong computable reducibility.
\newblock {\em Proc. Amer. Math. Soc.}, to appear.

\bibitem{Erdos1952Combinatorial}
Paul Erdos and Richard Rado.
\newblock Combinatorial theorems on classifications of subsets of a given set.
\newblock {\em Proceedings of the London mathematical Society}, 3(1):417--439,
  1952.

\bibitem{FP-TA}
Emanuele Frittaion and Ludovic Patey.
\newblock Coloring the rationals in reverse mathematics.
\newblock {\em Computability}, to appear.

\bibitem{Hirschfeldt-2014}
Denis~R. Hirschfeldt.
\newblock {\em Slicing the Truth: On the Computable and Reverse Mathematics of
  Combinatorial Principles}.
\newblock Lecture notes series / Institute for Mathematical Sciences, National
  University of Singapore. World Scientific Publishing Company Incorporated,
  2014.

\bibitem{HS-2007}
Denis~R. Hirschfeldt and Richard~A. Shore.
\newblock Combinatorial principles weaker than {R}amsey's theorem for pairs.
\newblock {\em J. Symbolic Logic}, 72(1):171--206, 2007.

\bibitem{Hirst-1987}
Jeffry~L. Hirst.
\newblock {\em Combinatorics in Subsystems of Second Order Arithmetic}.
\newblock PhD thesis, The Pennsylvania State University, 1987.

\bibitem{Jockusch-1972}
Carl~G. Jockusch, Jr.
\newblock {R}amsey's theorem and recursion theory.
\newblock {\em J. Symbolic Logic}, 37:268--280, 1972.

\bibitem{JS-1972a}
Carl~G. Jockusch, Jr. and Robert~I. Soare.
\newblock {$\Pi \sp{0}\sb{1}$} classes and degrees of theories.
\newblock {\em Trans. Amer. Math. Soc.}, 173:33--56, 1972.

\bibitem{KM-TA}
Mushfeq Khan and Joseph~S. Miller.
\newblock Forcing with bushy trees.
\newblock to appear.

\bibitem{Kumabe-1996}
Masahiro Kumabe.
\newblock A fixed point free minimal degree.
\newblock 1996.

\bibitem{KL-2009}
Masahiro Kumabe and Andrew E.~M. Lewis.
\newblock A fixed-point-free minimal degree.
\newblock {\em J. Lond. Math. Soc. (2)}, 80(3):785--797, 2009.

\bibitem{LST-2013}
Manuel Lerman, Reed Solomon, and Henry Towsner.
\newblock Separating principles below {R}amsey's theorem for pairs.
\newblock {\em J. Math. Log.}, 13(2):1350007, 44, 2013.

\bibitem{Liu-2012}
Jiayi Liu.
\newblock {$RT^2_2$ does not imply $WKL_0$}.
\newblock {\em J. Symbolic Logic}, 77(2):609--620, 2012.

\bibitem{Liu-2015}
Jiayi Liu.
\newblock Cone avoiding closed sets.
\newblock {\em Trans. Amer. Math. Soc.}, 367(3):1609--1630, 2015.

\bibitem{Mileti-2004}
Joseph~R. Mileti.
\newblock {\em Partition Theorems and Computability Theory}.
\newblock PhD thesis, University of Illinois at Urbana-Champaign, 2004.

\bibitem{Montalban-2011}
Antonio Montalb{\'a}n.
\newblock Open questions in reverse mathematics.
\newblock {\em Bull. Symbolic Logic}, 17(3):431--454, 2011.

\bibitem{MR3416155}
Ludovic Patey.
\newblock Ramsey-type graph coloring and diagonal non-computability.
\newblock {\em Arch. Math. Logic}, 54(7-8):899--914, 2015.

\bibitem{Patey2016reverse}
Ludovic Patey.
\newblock {\em The reverse mathematics of Ramsey-type theorems}.
\newblock PhD thesis, Université Paris Diderot, 2016.

\bibitem{Patey-TA4}
Ludovic Patey.
\newblock Somewhere over the rainbow ramsey theorem for pairs.
\newblock to appear.

\bibitem{Patey-TA3}
Ludovic Patey.
\newblock The strength of the tree theorem for pairs in reverse mathematics.
\newblock {\em J. Symbolic Logic}, to appear.

\bibitem{Patey-TA}
Ludovic Patey.
\newblock The weakness of being cohesive, thin or free in reverse mathematics.
\newblock {\em Israel J. Math.}, to appear.

\bibitem{SS-1995}
David Seetapun and Theodore~A. Slaman.
\newblock On the strength of {R}amsey's theorem.
\newblock {\em Notre Dame J. Formal Logic}, 36(4):570--582, 1995.
\newblock Special Issue: Models of arithmetic.

\bibitem{Shoenfield-1959}
J.~R. Shoenfield.
\newblock On degrees of unsolvability.
\newblock {\em Ann. of Math. (2)}, 69:644--653, 1959.

\bibitem{Shore-TA}
Richard~A. Shore.
\newblock Lecture notes on turing degrees.
\newblock In {\em Computational {P}rospects of {I}nfinity {II}: {AII}
  {G}raduate {S}ummer {S}chool}, Lect. Notes Ser. Inst. Math. Sci. Natl. Univ.
  Singap. World Sci. Publ., Hackensack, NJ, to appear.

\bibitem{Simpson-2009}
Stephen~G. Simpson.
\newblock {\em Subsystems of second order arithmetic}.
\newblock Perspectives in Logic. Cambridge University Press, Cambridge, second
  edition, 2009.

\bibitem{Simpson-1977}
Stephen G. Simpson~G. Simpson.
\newblock Degrees of unsolvability: A survey of results.
\newblock In J.~Barwise, editor, {\em Handbook of Mathematical Logic}, pages
  631--652. North-Holland, Amsterdam, 1977.

\bibitem{Soare-2016}
Robert~I. Soare.
\newblock {\em Turing Computability: Theory and Applications}.
\newblock Springer Publishing Company, Incorporated, 1st edition, 2016.

\end{thebibliography}

\end{document}